\documentclass{TPmod}
\usepackage{url}
\usepackage{setspace}
\usepackage{scrextend}
\usetikzlibrary{arrows.meta}

\usetikzlibrary{shapes.misc}
\usetikzlibrary{decorations.pathmorphing}

\usepackage{hyperref}
\usepackage{amsthm}
\usepackage{amssymb}
\usepackage[capitalize]{cleveref}
\usepackage{mathtools}

\newcommand{\Sym}{\mathrm{Sym}}

\newcommand{\st}{\mathrm{st}}

\def\scrA{\EuScript{A}}
\def\scrJ{\EuScript{J}}

\newcommand{\scrF}{\EuScript{F}}
\newcommand{\scrC}{\EuScript{C}}

\def\bK{\mathbb{K}}
\def\bP{\mathbb{P}}
\def\bA{\mathbb{A}}

\def\bZ{\mathbb{Z}}
\def\bQ{\mathbb{Q}}

\def\bC{\mathbb{C}}
\newcommand{\bS}{\mathbb{S}}

\def\bR{\mathbb{R}}

\def\calL{\mathcal{L}}

\def\scrD{\EuScript{D}}
\def\scrL{\EuScript{L}}
\def\scrG{\EuScript{G}}

\title{Floer theory of higher rank quiver 3-folds}
\author{Ivan Smith}
\date{January 2020}

\address{Centre for Mathematical Sciences, University of Cambridge, Wilberforce Road, CB3 0WB, U.K.}
\email{is200@cam.ac.uk }

\begin{abstract} {\sc Abstract:} 
We study threefolds $Y$ fibred by $A_m$-surfaces over a curve $S$ of positive genus.  An ideal triangulation of $S$ defines, for each rank $m$, a quiver $Q(\Delta_m)$, hence a $CY_3$-category $\scrC(W)$ for any potential $W$ on $Q(\Delta_m)$. We  show that for $\omega$ in an open subset of the K\"ahler cone, a subcategory of a sign-twisted Fukaya category of $(Y,\omega)$ is quasi-isomorphic to $(\scrC,W_{[\omega]})$ for a certain generic potential $W_{[\omega]}$. This partially establishes a conjecture of Goncharov \cite{Goncharov} concerning `categorifications' of cluster varieties of framed $\bP GL_{m+1}$-local systems on $S$, and gives a symplectic geometric viewpoint on results of Gaiotto, Moore and Neitzke \cite{GMN:snakes} on `theories of class $\mathcal{S}$'. \end{abstract}

\begin{document}
\maketitle

\section{Introduction}

Fix a pair of positive integers $(g,d)$ and consider a closed surface $\bS$ of genus $g>0$ with a non-empty collection of $d> 0$ marked points $\bP \subset \bS$.   Fix  in addition a positive integer $m> 0$, called the `rank'.  We work over the characteristic zero field $\bK = \Lambda_{\bC}$ which is the one-variable Novikov field over $\bC$. Associated to this data, there is
\begin{enumerate}
\item a $\bK$-linear CY$_3$ $A_{\infty}$-category $(\scrC, W)$,  obtained from a choice of potential $W$ on a quiver $Q(\Delta_m)$  associated to a choice of ideal triangulation $\Delta$ of $\bS$ with vertices at $\bP$ and with no self-folded triangles, see \cite{Ginzburg,Goncharov} and Section \ref{Sec:Quivers};
\item a non-compact K\"ahler Calabi-Yau threefold $(Y,\omega)$, which is the total space of a fibration by $A_m$-surfaces over $\bS$ with fibres over $\bP$ being disjoint unions of $m+1$ planes $\bC^2$, and a further collection of Lefschetz singular fibres, see \cite{Abrikosov,DDP} and Section \ref{Sec:quiver_3folds}.
\end{enumerate}
Both the above depend on choices:  the potential $W$ on the quiver (up to gauge equivalence) in the first case, and the (cohomology class of) K\"ahler form $\omega$ in the second.  We will make a further choice, which is an ordering of the $m+1$ connected components of the reducible singular fibre of $Y$ over each point of $\bP$.  The sum of  the even-indexed components in this ordering, summed over each such reducible fibre, defines a class $b\in H^2(Y;\bZ/2)$. Relative to this background class, there is a sign-twisted Fukaya category $\scrF((Y,\omega); b)$, which is an $A_{\infty}$-category over $\bK$ whose objects are  $b$-relatively spin graded $\omega$-Lagrangian submanifolds equipped with suitable brane data.   (The choice of cycle representative for $b$ is natural in a particular setting encountered later, but monodromy considerations show that up to quasi-isomorphism the category only depends on the number of components at each $p \in \bP$.)

\begin{thm}\label{thm:main}
There is a non-empty open subset $U\subset H^2(Y;\bR)$ of the K\"ahler cone, and a map $U \to \{\mathrm{potentials}\}/\{\mathrm{gauge}\}$, $[\omega] \mapsto W_{[\omega]}$, such that for $[\omega] \in U$ there is a fully faithful embedding $(\scrC, W_{[\omega]}) \hookrightarrow \scrF((Y,\omega);b)$.
\end{thm}

The hypothesis $g(\bS)>0$ simplifies the holomorphic curve theory (it implies the Fukaya category can be constructed using classical transversality theory).  After passing to twisted complexes on both sides, the image of the embedding of Theorem \ref{thm:main}  is a split-closed triangulated subcategory. We conjecture that image co-incides with the full subcategory of $\mathrm{Tw}\,\scrF(Y,\omega;b)$ generated by Lagrangian spheres, and is therefore intrinsic to the symplectic topology of $(Y,\omega)$. One could view the algebraic model $(\scrC, W_{[\omega]})$ as a `non-commutative mirror' to $(Y,\omega)$, and Theorem \ref{thm:main} as a statement of homological mirror symmetry in this setting\footnote{A  related result for $A_m$-fibrations over $\bS = \bC$, but concerning derived categories rather than Fukaya categories, appears in \cite{Velez-Boer}.}.

Goncharov conjectured in \cite[Conjecture 6.2]{Goncharov} that the $CY_3$-category associated to $Q(\Delta_m)$ and the `canonical' potential  $W = W(\Delta_m)$ on the underlying bipartite graph (as introduced in \cite{Franco-Hanany-etal}) should be realised as a subcategory of a Fukaya category.  Goncharov's conjecture, stemming from general expectations around `categorifications' of cluster varieties, was  futher elaborated by Abrikosov \cite[Conjecture 1.4]{Abrikosov}; Theorem \ref{thm:main}  proves the formulation given there.   The result also relates to questions of Shende, Treumann and Williams \cite[Problems 1.15 \& 1.16]{STW} on the existence of potentials governing local Calabi-Yau 3-folds associated to surfaces. 
 It should be possible to relate the canonical potential, which defines a $\bC$-linear category,  to the one coming from symplectic topology when the surface $\bS$ is punctured and the associated  threefold is exact as a symplectic manifold (this is true in the simplest case when there are punctures and no marked points / reducible fibres, in which case the canonical potential has only cubic and quartic terms; one can more generally  work in a  `relative Fukaya category', cf. Remark \ref{rmk:relative}).    In the non-exact case, holomorphic curves are weighted by their areas encoded in the Novikov variable; the resulting potential never has trivial Novikov valuation.  One can still recover the cohomology class $[\omega]$ from the potential, cf. Remark \ref{rmk:class_of_omega}.

The theorem is proved, following \cite{Smith:quiver} for $m=1$, by finding a collection of Lagrangian 3-spheres $\{L_v \, | \, v\in \mathrm{Vert}(Q(\Delta_m))\}$ in $Y$ whose Floer cohomology algebra $\oplus_{v,v'} HF^*(L_v,L_{v'})$  agrees with the Koszul dual to the Ginzburg algebra associated to $Q(\Delta_m)$.  (The open subset $U = U(\Delta) \subset H^2(Y;\bR)$ of the K\"ahler cone,  which in principle depends on $\Delta$,  is any for which the relevant configuration of Lagrangian spheres exists.   It is not clear if $\cup_{\Delta} U(\Delta)$ covers the K\"ahler cone.)   The general theory of cyclic $A_{\infty}$-structures then implies that the subcategory $\scrF(\scrL) \subset \scrF(Y;b)$ generated by $\{L_v\}$ is governed by some potential $W_{[\omega]}$ on $Q(\Delta_m)$. Further study of non-vanishing holomorphic polygon counts shows that $W_{[\omega]} = W_{\bf c}(\Delta_m) + W'$ for a $\bK$-coefficient vector ${\bf c}$  recording areas of polygons associated to certain distinguished `primitive' (chordless) cycles, and some `nonlocal' terms $W'$, which cannot \emph{a priori} be controlled.  (The `canonical' potential $W(\Delta_m)$ is exactly  $W_{\bf c}(\Delta_m)$ for a vector of coefficients each of which is $\pm 1$;  in the non-exact case we record information on $[\omega]$ in ${\bf c}$.)
We conclude that some $A_{\infty}$-deformation of $(\scrC, W_{\bf c}(\Delta_m))$ embeds into the Fukaya category, without specifying exactly which; resolving this ambiguity is a version of fixing a mirror map.   

On the geometric side, the crucial new ingredient when passing from $m=1$ to $m>1$ is  the presence of `tripod' Lagrangian spheres in $A_m$-Milnor fibres, see Section \ref{Sec:Tripods}, and their appearance in the sphere configurations associated to $\Delta_m$.

Theorem \ref{thm:main} relates to work of Gaiotto, Moore and Neitzke \cite{GMN:spectral, GMN:snakes} on `theories of class $\mathcal{S}$', certain four-dimensional $\mathcal{N}=2$ field theories.  They relate the BPS  degeneracies of solitons in  such theories   to `spectral networks' on a Riemann surface equipped with a tuple of meromorphic differentials. In the rank $m=1$ case, this relates BPS states and saddle connections of meromorphic quadratic differentials \cite{GMN, BridgelandSmith}.  Long-standing expectations in both mathematics and physics suggest that the counting of BPS states should be formalised by counts of stable objects in triangulated categories such as Fukaya categories. The tripod spheres which enter  into the proof of Theorem \ref{thm:main} exactly correspond to the simplest spectral networks after saddle connections, see \cite[Figure 3]{GMN:spectral}. The possible \emph{embedded} graded Lagrangians in $Y_{\Phi}$ are constrained by results of  \cite{GanatraPomerleano}, and one only obtains connect sums of copies of $S^1\times S^2$ and $3$-tori.  It would be interesting to construct unobstructed immersed special Lagrangian representatives for more general spectral networks.

\begin{rmk}
One can have pairs of tripods which meet at all three feet,  and give rise to Lagrangian 3-spheres $L_0, L_1$ which meet at 3 transverse intersection points of equal Maslov grading, and bounding no holomorphic discs. The subcategory $\langle L_0, L_1 \rangle \subset \scrF(Y_{\Phi};b)$ is quasi-isomorphic to the Ginzburg category of the three-arrow Kronecker quiver. It then follows from \cite{Reineke}, see also \cite{Mainiero}, that there are classes $\eta \in K(\scrF(Y_{\Phi};b))$ for which the DT-invariant of $d\cdot \eta$ grows exponentially with $d$, a phenomenon that does not occur for the 3-folds in rank one \cite{BridgelandSmith}. (The associated field theories of class $\mathcal{S}$ have `wild BPS spectra' and `BPS giants'.)   The 3-folds $Y_{\Phi}$ for rank $m>1$ contain graded Lagrangian submanifolds diffeomorphic to $(S^1\times S^2) \# (S^1 \times S^2)$, obtained from Lagrange surgery $L_0 \# L_1$ on $L_0$ and $L_1$.  Irreducible modules over the based loop space $\Omega(L_0\# L_1)$ give rise to candidate stable objects to realise wild BPS states in the Fukaya category. \end{rmk}

\paragraph{Acknowledgements.}   Dmitry Tonkonog contributed several ideas at an early stage. Thanks to Mohammed Abouzaid, Efim Abrikosov, Tom Bridgeland, Andy Neitzke and Pietro Longhi for  helpful conversations, and to Sasha Goncharov for his interest.   The author is partially funded by a Fellowship from the Engineering and Physical Sciences Research Council, U.K.

\section{Quivers and potentials from ideal triangulations\label{Sec:Quivers}}

\subsection{Categories from quivers with potential}

 A well-known construction due to Ginzburg \cite{Ginzburg} associates to a quiver with potential $(Q,W)$  a 3-dimensional Calabi-Yau cyclic $A_{\infty}$-category $\scrC(Q,W)$.  The category $\scrC(Q,W)$ is the total $A_{\infty}$-endomorphism algebra of a collection of spherical objects $S_v$ indexed by the vertices $v\in Q_0$ of $Q$; it  is concentrated in degrees $0 \leq * \leq 3$, the degree one morphism spaces are based by the arrows $Q_1$ of $Q$, and the potential gives a cyclic encoding of the non-trivial $A_{\infty}$-products,  see \cite[Section 2]{Smith:quiver} for a summary of the construction.  We denote by $\scrD(Q,W)$ the corresponding derived category.  Mutations of $(Q,W)$ induce equivalences of the derived categories by \cite{Keller-Yang}.

\begin{lem}\label{lem:lagrangians_to_quiver}  Let $(X,\omega)$ be a symplectic manifold with a well-defined Fukaya category $\scrF(X)$. Suppose we have objects $\{L_v \in \scrF(X) \, | \, v \in Q_0\}$ for which  the total (cohomological) endomorphism algebras
\[
\oplus_{v, v' \in Q_0} \, HF^*(L_v, L_{v'}) \cong \oplus_{v,v' \in Q_0} \, \Hom_{\scrC}(S_v, S_{v'})
\]
are isomorphic as graded algebras over the semisimple ring $\oplus_{v\in Q_0} \bK_v$ (with idempotents the units of the objects $L_i$ respectively $S_i$). Then the full $A_{\infty}$-subcategory $\scrL \subset \scrF(X)$ generated by the $\{L_v\}$ is encoded up to quasi-isomorphism by a potential $W_{\scrL}$ on $Q$.
\end{lem}

\begin{proof} This is almost tautological. The $A_{\infty}$-structure on Floer cochains $\oplus_{i.j} CF^*(L_i,L_j)$ can always be taken to be strictly unital, since the reduced Hochschild complex is quasi-isomorphic to the full Hochschild complex over a field, and that strictly unital structure can be pushed to cohomology by homological perturbation. The fact that $\bR \subset \bK$ means that the $A_{\infty}$-structure can also be taken to be strictly cyclic (this holds by abstract theory over any characteristic zero field \cite{KontSoi}, but can be achieved for geometric reasons for fields $\bK \supset \bR$  \cite{Fukaya:cyclic}). The book-keeping in the Ginzburg construction then shows that the subcategory $\scrL$ is encoded by a cyclic potential on $Q$.  
\end{proof}

In the setting of Lemma \ref{lem:lagrangians_to_quiver}, the cubic terms in the potential encode the Floer product, which is well-defined; the higher order terms determine the higher $A_{\infty}$-products which depend on choices of almost complex structure and perturbation data.  Lemma \ref{lem:lagrangians_to_quiver} implies that, given a finite collection $\{L_v\, | \, v \in Q_0\}$ of Lagrangian rational homology spheres in a CY 3-fold for which the morphism spaces $HF^*(L_v, L_{v'})$ are concentrated in degrees $1,2$ for all $v\neq v'$, then the $A_{\infty}$-structure on $\oplus_{v,v'} HF^*(L_v,L_{v'})$ is encoded by a cyclic potential on the quiver with vertices $Q_0$ and arrow spaces $Q_1$ indexed by bases for $HF^1(L_v,L_{v'})$.

\subsection{Gauge transformations}

Potentials are called \emph{right-equivalent}  if they are related by an automorphism of the completed path algebra; right equivalent potentials $W$ and $W'$ on $Q$ yield quasi-isomorphic $A_{\infty}$-categories $\scrC(Q,W) \simeq \scrC(Q,W')$, cf. \cite{Ginzburg,KontSoi}. There are $A_{\infty}$-equivalences which do not arise from right equivalences, for instance ones acting non-trivially on cohomology, and ones arising from the canonical $\bK^*$-action on $A_{\infty}$-structures which rescales $m^k$ by $\lambda^{k-2}$.

 The group $\scrG$ of right equivalences of the completed path algebra decomposes as a semidirect product
\[
\scrG = \scrG^{un} \rtimes \scrG^{diag}
\]
where the second factor of `diagonal' automorphisms are those which arise from automorphisms of the vector space of arrows (as bimodules over the semisimple ring given by the idempotent lazy paths at the vertices), and the first factor of `unitriangular' automorphisms are those induced by maps from the arrow space into the subspace of paths of length $\geq 2$.  When the arrow space between any two vertices is at most one-dimensional, then $\scrG^{diag} \cong (\bK^*)^{|\mathrm{Vert}(Q)|}$ acts just by diagonally rescaling the arrows.

A quiver has a finite distinguished set of `chordless' cycles, see \cite{DWZ}. A potential is `primitive' if it is a combination of chordless cycles, and every chordless cycle appears with non-zero coefficient. A potential is  `generic' if its projection to the span of chordless cycles is primitive.  Using the fact that the set of chordless cycles is intrinsic to the quiver, \cite[Section 5.1]{Abrikosov} asserts that projection to the primitive part of a potential yields a $\scrG^{diag}$-equivariant projection
\begin{equation} \label{eqn:project_potential}
\{\mathrm{Generic \  potentials}\} / \scrG \longrightarrow \{\mathrm{Primitive \ potentials}\}/\scrG^{diag}.
\end{equation}
Return to the situation of Lemma \ref{lem:lagrangians_to_quiver}. Lagrangian 3-spheres will persist as Lagrangians under any sufficiently small deformation of the symplectic form $[\omega] \in U \subset H^2(X;\bR)$ on $X$, and (appealing to sufficient technology in the non-weakly-exact case) will remain unobstructed. Suppose furthermore that the Floer cohomologies $\oplus_{v,v'} HF^*(L_v,L_{v'})$ do not change, as graded $\bK$-vector spaces, as one varies the symplectic form in $U$.  Then one obtains maps
\begin{equation} \label{eq:mirror}
H^2(X;\bR) \supset U \longrightarrow \{\mathrm{Potentials}\} / \scrG \longrightarrow \{\mathrm{Primitive \ potentials}\}/\scrG^{diag}.
\end{equation}
The RH group above is \emph{a priori} finite dimensional, whilst the set $\{\mathrm{Potentials}\}/\scrG$ of all cyclic $A_{\infty}$-structures need not be. Our aim is to determine the composite map \eqref{eq:mirror}; giving its lift to  $\{\mathrm{Potentials}\} / \scrG$ is somewhat like finding a `mirror map', which we leave undetermined. 

\begin{rmk} The map \eqref{eq:mirror} is not a local isomorphism (the domain and codomain have different dimensions). For one thing, the coefficients of the potential -- which record areas of holomorphic polygons determined by areas of polygonal regions in the dual cellulation $\Delta_m^{\vee}$ -- are governed by $[\omega] \in H^2(Y_{\Phi}, \sqcup_v L_{v};\bR)$, whilst the Fukaya category $\scrF(Y_{\Phi})$ only depends on $[\omega] \in H^2(Y_{\Phi};\bR)$.  \end{rmk}

\subsection{Quivers with potential from triangulations}
We summarise some results from \cite{Abrikosov}. Take an ideal triangulation $\Delta$ of $\bS$ with vertices at $\bP \subset \bS$ and with no self-folded triangles (i.e. all triangles have three distinct edges).  We place $m$ vertices on each edge of the ideal triangulation, and then subdivide the triangulation (cf. Figure \ref{Fig:inscribed}, showing the case $m=4$) to obtain a new triangulation $\Delta_m$.  We view this as bicoloured as in Figure \ref{Fig:inscribed}, so each triangle of $\Delta$ now has inscribed within it $m(m+1)/2$ black triangles.   We then orient the edges of these inscribed black triangles as in Figure \ref{Fig:inscribed}; doing this for each ideal triangle in $\Delta$ yields a quiver drawn on the surface $\bS$, each vertex of which is one of those originally placed on $\Delta$. We denote the resulting quiver by $Q(\Delta_m)$; it depends on $\Delta$ and the choice of rank $m \geq 1$.  

\begin{rmk}\label{rmk:numerics}
An ideal triangulation of a surface of genus $g$ with $d$ marked points (vertices) has $6g-6+3d$ edges and $4g-4+2d$ faces. 
\end{rmk}

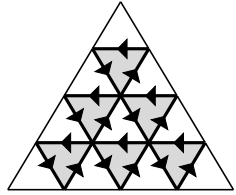
\begin{figure}[ht]
\begin{center}
\begin{tikzpicture}[scale=0.5]

\draw[semithick] (-3,0) -- (3,0);
\draw[semithick] (-3,0) -- (0,5);
\draw[semithick] (0,5) -- (3,0);

\draw[semithick] (-1.5,0) -- (-3 + 0.75, 1.25);
\draw[semithick] (0,0) -- (-3+1.5, 2.5);
\draw[semithick] (1.5,0) -- (-3+2.25, 3.75);

\draw[semithick] (-1.5,0) -- (3-2.25,3.75);
\draw[semithick] (0,0) -- (3-1.5,2.5);
\draw[semithick] (1.5,0) -- (3-0.75,1.25);

\draw[semithick] (-2.25,1.25) -- (2.25,1.25);
\draw[semithick] (-1.5,2.5) -- (1.5,2.5);
\draw[semithick] (-0.75,3.75) -- (0.75,3.75);

\draw[fill,gray!30] (-0.75,3.75) -- (0.75,3.75) -- (0, 2.5); 
\draw[fill,gray!30] (-1.5,2.5) -- (0,2.5) -- (-0.75, 1.25); 
\draw[fill,gray!30] (0,2.5)--(1.5,2.5) -- (0.75,1.25);
\draw[fill,gray!30] (-2.25,1.25) -- (-0.75,1.25) -- (-1.5,0);
\draw[fill,gray!30] (-0.75,1.25) -- (0.75,1.25) -- (0,0);
\draw[fill,gray!30] (0.75,1.25) -- (2.25,1.25) -- (1.5,0);

\newcommand{\midarrow}{\tikz \draw[-triangle 90] (0,0) -- +(.1,0);}

\begin{scope}[very thick, every node/.style={sloped,allow upside down}]

  \draw (0.75,3.75)-- node {\midarrow} (-0.75,3.75);
  \draw (-.75,3.75)-- node {\midarrow} (0,2.5);
  \draw (0,2.5) -- node {\midarrow} (0.75,3.75);
  
  \draw (0,2.5)-- node {\midarrow} (-1.5,2.5);
  \draw (-1.5,2.5)-- node {\midarrow} (-0.75,1.25);
  \draw (-.75,1.25) -- node {\midarrow} (0,2.5);
    
  \draw (0,2.5)-- node {\midarrow} (.75,1.25);
  \draw (.75,1.25)-- node {\midarrow} (1.5,2.5);
  \draw (1.5,2.5) -- node {\midarrow} (0,2.5);
   
   \draw (-1.5,0)-- node {\midarrow} (-0.75,1.25);
  \draw (-.75,1.25)-- node {\midarrow} (-2.25,1.25);
  \draw (-2.25,1.25) -- node {\midarrow} (-1.5,0);
  
     \draw (1.5,0)-- node {\midarrow} (2.25,1.25);
  \draw (2.25,1.25)-- node {\midarrow} (.75,1.25);
  \draw (.75,1.25) -- node {\midarrow} (1.5,0);

  \draw (-.75,1.25)-- node {\midarrow} (0,0);
 \draw (0,0)-- node {\midarrow} (0.75,1.25);
  \draw (.75,1.25)-- node {\midarrow} (-.75,1.25);

\end{scope}

\end{tikzpicture}
\end{center}
\caption{The inscribed quiver in one ideal triangle ($A_3$-case, i.e. for $Q(\Delta_2)$)\label{Fig:inscribed}}
\end{figure}

There are three visible collections of closed cycles on the inscribed quiver $Q(\Delta_m)$ (see Figures \ref{Fig:inscribed}, \ref{Fig:cycle_types_both} and Figure \ref{Fig:cycle_types_punctures}). 
We will call these `primitive cycles'. Namely, one has

\begin{figure}[ht]
\begin{center}
\begin{tikzpicture}[scale=0.8]
\newcommand{\midarrow}{\tikz \draw[-triangle 90] (0,0) -- +(.1,0);}

\draw[semithick,dashed] (0,0) -- (3,0);
\draw[semithick,dashed] (0,0) -- (-3,0);
\draw[semithick,dashed] (0,0) -- ({3*cos(60)}, {3*sin(60)});
\draw[semithick,dashed] (0,0) -- ({-3*cos 60}, {3*sin 60});
\draw[semithick,dashed] (0,0) -- ({3*cos 60}, {-3*sin 60});
\draw[semithick,dashed] (0,0) -- ({-3*cos 60}, {-3*sin 60});

\draw[semithick,dashed] ({-3*cos 60}, {3*sin 60}) -- ({3*cos 60}, {3*sin 60});
\draw[semithick,dashed] ({-3*cos 60}, {-3*sin 60}) -- ({3*cos 60}, {-3*sin 60});
\draw[semithick,dashed] ({-3*cos 60}, {-3*sin 60}) -- (-3,0);
\draw[semithick,dashed] ({3*cos 60}, {-3*sin 60}) -- (3,0);
\draw[semithick,dashed] ({3*cos 60}, {3*sin 60}) -- (3,0);
\draw[semithick,dashed] ({-3*cos 60}, {3*sin 60}) -- (-3,0);

\begin{scope}[very thick, every node/.style={sloped,allow upside down}]

\draw[semithick,blue] (1,0) -- node {\midarrow} ({1+cos 60}, {sin 60});
\draw[semithick] ({1+cos 60}, {sin 60}) -- node {\midarrow} ({cos 60}, {sin 60});
\draw[semithick,red] ({cos 60}, {sin 60}) -- node {\midarrow} (1,0);

\draw[semithick] (2,0) -- node {\midarrow} ({2+cos 60}, {sin 60});
\draw[semithick] ({2+cos 60}, {sin 60}) -- node {\midarrow} ({1+cos 60}, {sin 60});
\draw[semithick,blue] ({1+cos 60}, {sin 60}) -- node {\midarrow} (2,0);

\draw[semithick] ({1+cos 60}, {sin 60}) -- node {\midarrow} ({1+2*cos 60}, {2*sin 60});
\draw[semithick] ({1+2*cos 60}, {2*sin 60}) -- node {\midarrow} ({2*cos 60}, {2*sin 60});
\draw[semithick] ({2*cos 60}, {2*sin 60}) -- node {\midarrow} ({1+cos 60},{sin 60});

\draw[semithick,red] (-1,0) -- node {\midarrow} ({-1+cos 60}, {sin 60});
\draw[semithick] ({-1+cos 60}, {sin 60}) -- node {\midarrow} ({-2+cos 60}, {sin 60});
\draw[semithick] ({-2+cos 60}, {sin 60}) -- node {\midarrow} (-1,0);

\draw[semithick] (-2,0) -- node {\midarrow} ({-2+cos 60}, {sin 60});
\draw[semithick] ({-2+cos 60}, {sin 60}) -- node {\midarrow} ({-3+cos 60}, {sin 60});
\draw[semithick] ({-3+cos 60}, {sin 60}) -- node {\midarrow} (-2,0);

\draw[semithick] ({-2+cos 60}, {sin 60}) -- node {\midarrow} ({-2+2*cos 60}, {2*sin 60});
\draw[semithick] ({-2+2*cos 60}, {2*sin 60}) -- node {\midarrow} ({-3+2*cos 60}, {2*sin 60});
\draw[semithick] ({-3+2*cos 60}, {2*sin 60}) -- node {\midarrow} ({-2+cos 60},{sin 60});

\draw[semithick,red] ({-cos 60, sin 60}) -- node{\midarrow} ({cos 60}, {sin 60});
\draw[semithick] ({cos 60, sin 60}) -- node{\midarrow} (0, {2*sin 60});
\draw[semithick] (0, {2*sin 60}) -- node{\midarrow} ({-cos 60}, {sin 60});

\draw[semithick] (0,{2*sin 60}) -- node{\midarrow} (1, {2*sin 60});
\draw[semithick] (1, {2*sin 60}) -- node{\midarrow} ({cos 60},{3*sin 60});
\draw[semithick] ({cos 60, 3*sin 60}) -- node{\midarrow} (0, {2*sin 60});

\draw[semithick] (-1,{2*sin 60}) -- node{\midarrow} (0, {2*sin 60});
\draw[semithick] (0, {2*sin 60}) -- node{\midarrow} ({-cos 60},{3*sin 60});
\draw[semithick] ({-cos 60, 3*sin 60}) -- node{\midarrow} (-1, {2*sin 60});

\draw[semithick,red]  ({cos 60, -sin 60}) -- node{\midarrow} ({-cos 60}, {-sin 60});
\draw[semithick] ({-cos 60, -sin 60}) -- node{\midarrow} (0, {-2*sin 60});
\draw[semithick] (0, {-2*sin 60}) -- node{\midarrow} ({cos 60}, {-sin 60});

\draw[semithick] (0,{-2*sin 60}) -- node{\midarrow} (-1, {-2*sin 60});
\draw[semithick] (-1, {-2*sin 60}) -- node{\midarrow} ({-cos 60},{-3*sin 60});
\draw[semithick] ({-cos 60, -3*sin 60}) -- node{\midarrow} (0, {-2*sin 60});

\draw[semithick] (1,{-2*sin 60}) -- node{\midarrow} (0, {-2*sin 60});
\draw[semithick] (0, {-2*sin 60}) -- node{\midarrow} ({cos 60},{-3*sin 60});
\draw[semithick] ({cos 60, -3*sin 60}) -- node{\midarrow} (1, {-2*sin 60});

\draw[semithick,red] (1,0) -- node {\midarrow} ({1-cos 60}, {-sin 60});
\draw[semithick] ({1-cos 60}, {-sin 60}) -- node {\midarrow} ({2-cos 60}, {-sin 60});
\draw[semithick,blue] ({2-cos 60}, {-sin 60}) -- node {\midarrow} (1,0);

\draw[semithick,blue] (2,0) -- node {\midarrow} ({2-cos 60}, {-sin 60});
\draw[semithick] ({2-cos 60}, {-sin 60}) -- node {\midarrow} ({3-cos 60}, {-sin 60});
\draw[semithick] ({3-cos 60}, {-sin 60}) -- node {\midarrow} (2,0);

\draw[semithick] ({2-cos 60}, {-sin 60}) -- node {\midarrow} ({2-2*cos 60}, {-2*sin 60});
\draw[semithick] ({2-2*cos 60}, {-2*sin 60}) -- node {\midarrow} ({3-2*cos 60}, {-2*sin 60});
\draw[semithick] ({3-2*cos 60}, {-2*sin 60}) -- node {\midarrow} ({2-cos 60},{-sin 60});

\draw[semithick] (-1,0) -- node {\midarrow} ({-1-cos 60}, {-sin 60});
\draw[semithick] ({-1-cos 60}, {-sin 60}) -- node {\midarrow} ({-cos 60}, {-sin 60});
\draw[semithick,red] ({-cos 60}, {-sin 60}) -- node {\midarrow} (-1,0);

\draw[semithick] (-2,0) -- node {\midarrow} ({-2-cos 60}, {-sin 60});
\draw[semithick] ({-2-cos 60}, {-sin 60}) -- node {\midarrow} ({-1-cos 60}, {-sin 60});
\draw[semithick] ({-1-cos 60}, {-sin 60}) -- node {\midarrow} (-2,0);

\draw[semithick] ({-1-cos 60}, {-sin 60}) -- node {\midarrow} ({-1-2*cos 60}, {-2*sin 60});
\draw[semithick] ({-1-2*cos 60}, {-2*sin 60}) -- node {\midarrow} ({-2*cos 60}, {-2*sin 60});
\draw[semithick] ({-2*cos 60}, {-2*sin 60}) -- node {\midarrow} ({-1-cos 60},{-sin 60});

\end{scope}

\end{tikzpicture}
\end{center}
\caption{The cycle $L_p^{(1)}$ (red) and a quadrilateral $q_w$ (blue)\label{Fig:cycle_types_both}}
\end{figure}
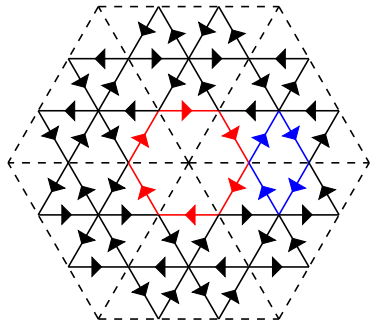

\begin{figure}[ht]
\begin{center}
\begin{tikzpicture}[scale=0.8]
\newcommand{\midarrow}{\tikz \draw[-triangle 90] (0,0) -- +(.1,0);}

\draw[semithick,dashed] (0,0) -- (3,0);
\draw[semithick,dashed] (0,0) -- (-3,0);
\draw[semithick,dashed] (0,0) -- ({3*cos(60)}, {3*sin(60)});
\draw[semithick,dashed] (0,0) -- ({-3*cos 60}, {3*sin 60});
\draw[semithick,dashed] (0,0) -- ({3*cos 60}, {-3*sin 60});
\draw[semithick,dashed] (0,0) -- ({-3*cos 60}, {-3*sin 60});

\draw[semithick,dashed] ({-3*cos 60}, {3*sin 60}) -- ({3*cos 60}, {3*sin 60});
\draw[semithick,dashed] ({-3*cos 60}, {-3*sin 60}) -- ({3*cos 60}, {-3*sin 60});
\draw[semithick,dashed] ({-3*cos 60}, {-3*sin 60}) -- (-3,0);
\draw[semithick,dashed] ({3*cos 60}, {-3*sin 60}) -- (3,0);
\draw[semithick,dashed] ({3*cos 60}, {3*sin 60}) -- (3,0);
\draw[semithick,dashed] ({-3*cos 60}, {3*sin 60}) -- (-3,0);

\begin{scope}[very thick, every node/.style={sloped,allow upside down}]

\draw[semithick] (1,0) -- node {\midarrow} ({1+cos 60}, {sin 60});
\draw[semithick] ({1+cos 60}, {sin 60}) -- node {\midarrow} ({cos 60}, {sin 60});
\draw[semithick,red] ({cos 60}, {sin 60}) -- node {\midarrow} (1,0);

\draw[semithick] (2,0) -- node {\midarrow} ({2+cos 60}, {sin 60});
\draw[semithick] ({2+cos 60}, {sin 60}) -- node {\midarrow} ({1+cos 60}, {sin 60});
\draw[semithick,red] ({1+cos 60}, {sin 60}) -- node {\midarrow} (2,0);

\draw[semithick] ({1+cos 60}, {sin 60}) -- node {\midarrow} ({1+2*cos 60}, {2*sin 60});
\draw[semithick] ({1+2*cos 60}, {2*sin 60}) -- node {\midarrow} ({2*cos 60}, {2*sin 60});
\draw[semithick,red] ({2*cos 60}, {2*sin 60}) -- node {\midarrow} ({1+cos 60},{sin 60});

\draw[semithick,red] (-1,0) -- node {\midarrow} ({-1+cos 60}, {sin 60});
\draw[semithick] ({-1+cos 60}, {sin 60}) -- node {\midarrow} ({-2+cos 60}, {sin 60});
\draw[semithick] ({-2+cos 60}, {sin 60}) -- node {\midarrow} (-1,0);

\draw[semithick,red] (-2,0) -- node {\midarrow} ({-2+cos 60}, {sin 60});
\draw[semithick] ({-2+cos 60}, {sin 60}) -- node {\midarrow} ({-3+cos 60}, {sin 60});
\draw[semithick] ({-3+cos 60}, {sin 60}) -- node {\midarrow} (-2,0);

\draw[semithick,red] ({-2+cos 60}, {sin 60}) -- node {\midarrow} ({-2+2*cos 60}, {2*sin 60});
\draw[semithick] ({-2+2*cos 60}, {2*sin 60}) -- node {\midarrow} ({-3+2*cos 60}, {2*sin 60});
\draw[semithick] ({-3+2*cos 60}, {2*sin 60}) -- node {\midarrow} ({-2+cos 60},{sin 60});

\draw[semithick,red] ({-cos 60, sin 60}) -- node{\midarrow} ({cos 60}, {sin 60});
\draw[semithick] ({cos 60, sin 60}) -- node{\midarrow} (0, {2*sin 60});
\draw[semithick] (0, {2*sin 60}) -- node{\midarrow} ({-cos 60}, {sin 60});

\draw[semithick,red] (0,{2*sin 60}) -- node{\midarrow} (1, {2*sin 60});
\draw[semithick] (1, {2*sin 60}) -- node{\midarrow} ({cos 60},{3*sin 60});
\draw[semithick] ({cos 60, 3*sin 60}) -- node{\midarrow} (0, {2*sin 60});

\draw[semithick,red] (-1,{2*sin 60}) -- node{\midarrow} (0, {2*sin 60});
\draw[semithick] (0, {2*sin 60}) -- node{\midarrow} ({-cos 60},{3*sin 60});
\draw[semithick] ({-cos 60, 3*sin 60}) -- node{\midarrow} (-1, {2*sin 60});

\draw[semithick,red]  ({cos 60, -sin 60}) -- node{\midarrow} ({-cos 60}, {-sin 60});
\draw[semithick] ({-cos 60, -sin 60}) -- node{\midarrow} (0, {-2*sin 60});
\draw[semithick] (0, {-2*sin 60}) -- node{\midarrow} ({cos 60}, {-sin 60});

\draw[semithick,red] (0,{-2*sin 60}) -- node{\midarrow} (-1, {-2*sin 60});
\draw[semithick] (-1, {-2*sin 60}) -- node{\midarrow} ({-cos 60},{-3*sin 60});
\draw[semithick] ({-cos 60, -3*sin 60}) -- node{\midarrow} (0, {-2*sin 60});

\draw[semithick,red] (1,{-2*sin 60}) -- node{\midarrow} (0, {-2*sin 60});
\draw[semithick] (0, {-2*sin 60}) -- node{\midarrow} ({cos 60},{-3*sin 60});
\draw[semithick] ({cos 60, -3*sin 60}) -- node{\midarrow} (1, {-2*sin 60});

\draw[semithick,red] (1,0) -- node {\midarrow} ({1-cos 60}, {-sin 60});
\draw[semithick] ({1-cos 60}, {-sin 60}) -- node {\midarrow} ({2-cos 60}, {-sin 60});
\draw[semithick] ({2-cos 60}, {-sin 60}) -- node {\midarrow} (1,0);

\draw[semithick,red] (2,0) -- node {\midarrow} ({2-cos 60}, {-sin 60});
\draw[semithick] ({2-cos 60}, {-sin 60}) -- node {\midarrow} ({3-cos 60}, {-sin 60});
\draw[semithick] ({3-cos 60}, {-sin 60}) -- node {\midarrow} (2,0);

\draw[semithick,red] ({2-cos 60}, {-sin 60}) -- node {\midarrow} ({2-2*cos 60}, {-2*sin 60});
\draw[semithick] ({2-2*cos 60}, {-2*sin 60}) -- node {\midarrow} ({3-2*cos 60}, {-2*sin 60});
\draw[semithick] ({3-2*cos 60}, {-2*sin 60}) -- node {\midarrow} ({2-cos 60},{-sin 60});

\draw[semithick] (-1,0) -- node {\midarrow} ({-1-cos 60}, {-sin 60});
\draw[semithick] ({-1-cos 60}, {-sin 60}) -- node {\midarrow} ({-cos 60}, {-sin 60});
\draw[semithick,red] ({-cos 60}, {-sin 60}) -- node {\midarrow} (-1,0);

\draw[semithick] (-2,0) -- node {\midarrow} ({-2-cos 60}, {-sin 60});
\draw[semithick] ({-2-cos 60}, {-sin 60}) -- node {\midarrow} ({-1-cos 60}, {-sin 60});
\draw[semithick,red] ({-1-cos 60}, {-sin 60}) -- node {\midarrow} (-2,0);

\draw[semithick] ({-1-cos 60}, {-sin 60}) -- node {\midarrow} ({-1-2*cos 60}, {-2*sin 60});
\draw[semithick] ({-1-2*cos 60}, {-2*sin 60}) -- node {\midarrow} ({-2*cos 60}, {-2*sin 60});
\draw[semithick,red] ({-2*cos 60}, {-2*sin 60}) -- node {\midarrow} ({-1-cos 60},{-sin 60});

\end{scope}

\end{tikzpicture}
\end{center}
\caption{The cycles $L_p^{(1)}$ and $L_p^{(2)}$ (red)\label{Fig:cycle_types_punctures}}
\end{figure}
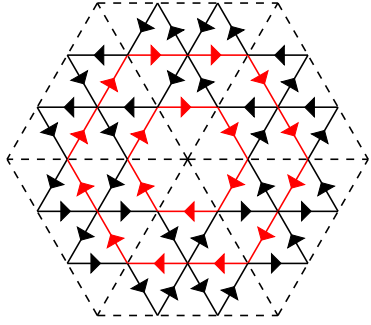
\begin{itemize}
\item anticlockwise-oriented 3-cycles $\{t_b\}$, each the boundary of a single  inscribed black triangle $b$;
\item clockwise-oriented 3- and 4-cycles $\{q_w\}$, boundaries of the white regions  $w$ which are those complementary regions on $\bS$ to the black triangles which do not contain a point $p\in \bP$; see Figure \ref{Fig:cycle_types_both} for a four-cycle $q_w$;
\item  for each point $p\in \bP$, which has valence $k$ as a vertex of $\Delta$, larger clockwise-oriented $k$-cycles $L_p^{(j)}$ for $1\leq j \leq m$; see Figure \ref{Fig:cycle_types_punctures}. 
\end{itemize}
When $m=1$, the middle class is not present
; when $m=2$, there are only quadrilaterals in the middle class, and no 3-cycles.  

\begin{rmk} The primitive cycles are exactly the chordless cycles for $Q(\Delta_m)$. \end{rmk}

The decomposition of $\bS$ into the black and white regions of the quiver and its complement amounts to giving a bipartite graph on $\bS$, and leads to a `canonical' potential $Q(\Delta_m)$, originating in the string theory community \cite{Franco-Hanany-etal} and emphasised in this setting by Goncharov in \cite{Goncharov}.  We will write $N$ for the total number of primitive cycles, so $N = [(4g-4+2d)m(m+1)/2] +[(6g-6+3d)(m-2) + (4g-4+2d)(m-2)(m-1)/2]+ [dm]$ for the numbers of $t_b, q_w, L_p^{(j)}$ respectively. 

\begin{defn}
For a vector ${\bf c} \in (\bK^*)^N$ of pointwise non-zero coefficients, we will write $W_{\bf c}(\Delta_m) = \sum c_b \cdot t_b + \sum c_w \cdot q_w + \sum_{p,j} c_p^{(j)} L_p^{(j)}$. 
\end{defn}

 In the $m=1$ case, and assuming $|\bP| > 1$, \cite{GLFS} show that every generic potential is right-equivalent to a generic primitive potential, i.e. one with zero non-primitive part.  This fails when $m>1$: then  the moduli space of $A_{\infty}$-structures on the cohomological category underlying $\scrC(Q(\Delta_m))$ has positive dimension.  

\begin{rmk}\label{rmk:abrikosov} By  diagonal automorphisms, any generic potential can be related to a `normalised' one in which the coefficients of all primitive cycles other than the $L_p^{(j)}$ are equal to $1$. If $m=2$, a normalised generic potential $W$ is \emph{strongly generic} if,  for each of the points $p\in \bP$, the (necessarily non-zero) coefficients $c_p^{(1)}$ and $c_p^{(2)}$ of $L_p^{(1)}$ and $L_p^{(2)}$ in $W$ satisfy the non-degeneracy condition
\begin{equation} \label{eqn:strongly_generic}
c_p^{(1)} + (-1)^{\mathrm{valence}(p)} c_p^{(2)} \neq 0.
\end{equation}
Right-equivalence preserves generic potentials, and the subset of those generic potentials which when normalised satisfy strong genericity.  The space of strongly generic potentials with fixed primitive part, up to right equivalence, is isomorphic to  $\bA^1_{\bK}$, cf.  \cite[Theorem 2]{Abrikosov} and \cite[Proposition 5.14]{Abrikosov}. When $m>2$, there is no explicit description of the generic fibre of \eqref{eqn:project_potential}. \end{rmk}

Any two ideal triangulations $\Delta$ and $\Delta'$ of $(\bS,\bP)$  (without self-folded triangles) can be related by a sequence of flips. Goncharov \cite{Goncharov} showed that the effect of a flip on the pair $(Q(\Delta_m), W(\Delta_m))$ could itself be effected by a sequence of $m(m+1)(m+2)/6$  mutations,  and that the mutation of the canonical potential is right-equivalent to the canonical potential on the mutated quiver. Mutations induce auto-equivalences of the associated $CY_3$-categories \cite{Keller-Yang}.  It follows that there is a well-defined $CY_3$-category $\scrD(\bS,\bP,m)$, quasi-isomorphic to the derived category of $\scrC(Q(\Delta_m), W(\Delta_m))$ for any choice of ideal triangulation $\Delta$ of $(\bS,\bP)$.  More generally, the family of $CY_3$-categories associated to all possible generic potentials can be realised by generic potentials on a fixed quiver $Q(\Delta_m)$.

\section{Quiver 3-folds and  Lagrangian sphere configurations\label{Sec:quiver_3folds}}

\subsection{$A_m$-fibred 3-folds} 
In this section we discuss the symplectic topology of the threefolds $Y(\bS,\bP,m)$. These threefolds were introduced in \cite{Abrikosov}, and are associated to tuples of meromorphic differentials $(\phi_2,\ldots,\phi_{m+1})$ on a Riemann surface $S$ underlying $\bS$ with poles at a subset $D\subset S$ of cardinality $|\bP|$; thus $\phi_j \in H^0(K_S(D)^{\otimes j})$. The threefolds associated to tuples of holomorphic differentials were previously introduced in \cite{DDP}, and those associated to meromorphic quadratic differentials in \cite{Smith:quiver}. 

Fix a Riemann surface $S$ of genus $g$, and a section $\delta \in H^0(\mathcal{O}_S(D))$ which vanishes to order 1 at a divisor $D$ of degree $d$ (we think of $D$ as lying at the points of $\bP \subset \bS$, where $\bS$ is the topological surface underlying the Riemann surface $S$). Note that $\delta$ is unique up to scale. We also fix a decomposition of the log canonical bundle
\begin{equation} \label{eqn:log_canonical}
K_S(D) = \calL_1 \otimes \calL_2.
\end{equation} 
We consider the rank 3 vector bundle
\begin{equation} \label{eqn:ambient}
\mathcal{W} = \calL_1^{\otimes(m+1)}(-D) \oplus \calL_1 \calL_2 \oplus \calL_2^{\otimes(m+1)}
\end{equation}
over $S$. Given a tuple 
\[
\Phi = (\phi_2,\ldots,\phi_{m+1}) \quad \mathrm{with}  \ \phi_j \in H^0(K_S(D)^{\otimes j})
\]
we consider the hypersurface
\[
Y_{\Phi} = \left\{ (a,b,c) \in \mathcal{W} \, \big| \, (\delta\cdot a)\cdot c = b^{m+1} - \sum_j b^{m+1-j} \cdot \phi_j \right\}
\]
Here $(a,b,c)$ are written with respect to the decomposition \eqref{eqn:ambient} of the rank 3 bundle $\mathcal{W}$. 
The terms $(\delta\cdot a)c$ and $b^{m+1} - \sum_j b^{m+1-j} \cdot \phi_j $ both belong to $K_S(D)^{\otimes(m+1)}$, so the defining equation makes sense.  (We ask that the sum of the roots of $\Phi(b) = 0$ vanishes for compatibility and comparison with \cite{GMN:snakes}.)

\begin{lem} $Y_{\Phi}$ has vanishing canonical class, so is a quasi-projective Calabi-Yau variety.
\end{lem}

\begin{proof} See \cite[Section 6]{Abrikosov}. \end{proof}

The \emph{spectral curve} $\Sigma \subset \mathrm{Tot}(K_S(D))$ is the vanishing locus $\{b \, | \, \Phi(b) = 0\}$. We say that $\Phi$ is generic when $\Sigma$ is smooth and projection $\Sigma \to S$ is a simple branched covering; it then has covering degree $m+1$ and $m(m+1)(2g-2+d)$ simple branch points, arising from the zeroes of $\det(\Phi)$.  The threefold is almost a conic $\bC^*$-bundle over $\mathrm{Tot}(K_S(D))$ with singular fibres $\bC\vee\bC$ along the spectral curve $\Sigma$: this is precisely true after a finite number of affine modifications of the bundle $K_S(D)$ at the fibres over points of $D$, see \cite[Section 6.3]{Abrikosov}.

We will work with K\"ahler forms on $Y_{\Phi}$ which are small perturbations of those induced from a choice of K\"ahler form on $S$ and on the total space of $\mathcal{W} \to S$; in particular our K\"ahler forms tame integrable complex structures for which the projection $Y_{\Phi} \to S$ is holomorphic.  Because the defining equation for $Y_{\Phi}$ is weighted homogeneous, parallel transport vector fields have polynomial growth on the fibres with respect to a K\"ahler metric on $Y_{\Phi}$  induced from a metric on the vector bundle $\mathcal{W}$, and there are globally defined parallel transport maps of the fibres of $Y_{\Phi}$ over paths in $S$.  Furthermore, there are parallel transport maps defined on compact subsets of $Y_{\Phi}$ over compact subsets in the universal family of threefolds obtained by varying $\Phi$.   (Neither the monodromy of $Y_{\Phi}\to S$, nor of the universal family, is naturally compactly supported; in the former case this is because if one compactifies the fibration vertically, the divisor at infinity is not locally trivial but degenerates over $D$.)

Given an ideal triangulation of $\bS$ with vertices at $\bP$, and its inscribed quiver, we pass to the dual graph of the quiver, as in Figures \ref{Fig:a3_lag_cellulation} and \ref{Fig:dual_lagrangians}.

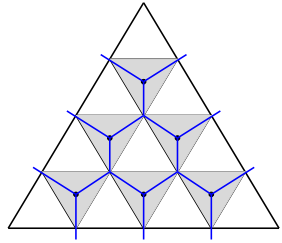
\begin{figure}[ht]
\begin{center}
\begin{tikzpicture}[scale=0.6]

\draw[semithick] (-3,0) -- (3,0);
\draw[semithick] (-3,0) -- (0,5);
\draw[semithick] (0,5) -- (3,0);

\draw[semithick] (-1.5,0) -- (-3 + 0.75, 1.25);
\draw[semithick] (0,0) -- (-3+1.5, 2.5);
\draw[semithick] (1.5,0) -- (-3+2.25, 3.75);

\draw[semithick] (-1.5,0) -- (3-2.25,3.75);
\draw[semithick] (0,0) -- (3-1.5,2.5);
\draw[semithick] (1.5,0) -- (3-0.75,1.25);

\draw[semithick] (-2.25,1.25) -- (2.25,1.25);
\draw[semithick] (-1.5,2.5) -- (1.5,2.5);
\draw[semithick] (-0.75,3.75) -- (0.75,3.75);

\draw[fill,gray!30] (-0.75,3.75) -- (0.75,3.75) -- (0, 2.5); 
\draw[fill,gray!30] (-1.5,2.5) -- (0,2.5) -- (-0.75, 1.25); 
\draw[fill,gray!30] (0,2.5)--(1.5,2.5) -- (0.75,1.25);
\draw[fill,gray!30] (-2.25,1.25) -- (-0.75,1.25) -- (-1.5,0);
\draw[fill,gray!30] (-0.75,1.25) -- (0.75,1.25) -- (0,0);
\draw[fill,gray!30] (0.75,1.25) -- (2.25,1.25) -- (1.5,0);

\draw[fill] (0,3/4) circle (0.05);
\draw[fill] (0, 13/4) circle (0.05);
\draw[fill] (-1.5, 3/4) circle (0.05);
\draw[fill] (1.5,3/4) circle (0.05);
\draw[fill] (-.75, 2) circle (0.05);
\draw[fill] (.75,2) circle (0.05);

\draw[semithick,blue] (0,3/4)--(-0.75,1.25);
\draw[semithick,blue] (0,3/4) -- (0.75,1.25);
\draw[semithick,blue] (0,3/4) -- (0, -1/4);
\draw[semithick,blue] (-1.5,3/4) -- (-1.5,-1/4);
\draw[semithick,blue] (-1.5,3/4) -- (-0.75,1.25);
\draw[semithick,blue] (-1.5,3/4) -- (-2.25-1/5,1.25+1/10);
\draw[semithick,blue] (1.5,3/4) -- (1.5,-1/4);
\draw[semithick,blue] (1.5,3/4) -- (0.75,1.25);
\draw[semithick,blue] (1.5,3/4) -- (2.25+1/5,1.25+1/10);
\draw[semithick,blue] (-0.75,2) -- (0,2.5);
\draw[semithick,blue] (-0.75,2) -- (-0.75,1.25);
\draw[semithick,blue] (-0.75,2) -- (-1.5-1/5,2.5+1/10);
\draw[semithick,blue] (0.75,2) -- (0,2.5);
\draw[semithick,blue] (0.75,2) -- (0.75,1.25);
\draw[semithick,blue] (0.75,2) -- (1.5+1/5,2.5+1/10);
\draw[semithick,blue] (0,13/4) -- (0,2.5);
\draw[semithick,blue] (0,13/4) -- (-0.75-1/5,3.75+1/10);
\draw[semithick,blue] (0,13/4) -- (0.75+1/5,3.75+1/10);

\end{tikzpicture}
\end{center}
\caption{The dual Lagrangian cellulation in one ideal triangle for $\Delta_3$\label{Fig:a3_lag_cellulation}}
\end{figure}

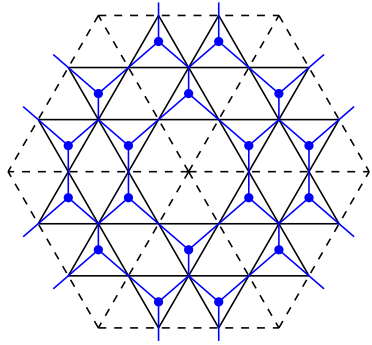
\begin{figure}[ht]
\begin{center}
\begin{tikzpicture}[scale=0.8]
\newcommand{\midarrow}{}

\draw[semithick,dashed] (0,0) -- (3,0);
\draw[semithick,dashed] (0,0) -- (-3,0);
\draw[semithick,dashed] (0,0) -- ({3*cos(60)}, {3*sin(60)});
\draw[semithick,dashed] (0,0) -- ({-3*cos 60}, {3*sin 60});
\draw[semithick,dashed] (0,0) -- ({3*cos 60}, {-3*sin 60});
\draw[semithick,dashed] (0,0) -- ({-3*cos 60}, {-3*sin 60});

\draw[semithick,dashed] ({-3*cos 60}, {3*sin 60}) -- ({3*cos 60}, {3*sin 60});
\draw[semithick,dashed] ({-3*cos 60}, {-3*sin 60}) -- ({3*cos 60}, {-3*sin 60});
\draw[semithick,dashed] ({-3*cos 60}, {-3*sin 60}) -- (-3,0);
\draw[semithick,dashed] ({3*cos 60}, {-3*sin 60}) -- (3,0);
\draw[semithick,dashed] ({3*cos 60}, {3*sin 60}) -- (3,0);
\draw[semithick,dashed] ({-3*cos 60}, {3*sin 60}) -- (-3,0);

\begin{scope}[very thick, every node/.style={sloped,allow upside down}]

\draw[semithick] (1,0) -- node {\midarrow} ({1+cos 60}, {sin 60});
\draw[semithick] ({1+cos 60}, {sin 60}) -- node {\midarrow} ({cos 60}, {sin 60});
\draw[semithick] ({cos 60}, {sin 60}) -- node {\midarrow} (1,0);

\draw[semithick] (2,0) -- node {\midarrow} ({2+cos 60}, {sin 60});
\draw[semithick] ({2+cos 60}, {sin 60}) -- node {\midarrow} ({1+cos 60}, {sin 60});
\draw[semithick] ({1+cos 60}, {sin 60}) -- node {\midarrow} (2,0);

\draw[semithick] ({1+cos 60}, {sin 60}) -- node {\midarrow} ({1+2*cos 60}, {2*sin 60});
\draw[semithick] ({1+2*cos 60}, {2*sin 60}) -- node {\midarrow} ({2*cos 60}, {2*sin 60});
\draw[semithick] ({2*cos 60}, {2*sin 60}) -- node {\midarrow} ({1+cos 60},{sin 60});

\draw[semithick] (-1,0) -- node {\midarrow} ({-1+cos 60}, {sin 60});
\draw[semithick] ({-1+cos 60}, {sin 60}) -- node {\midarrow} ({-2+cos 60}, {sin 60});
\draw[semithick] ({-2+cos 60}, {sin 60}) -- node {\midarrow} (-1,0);

\draw[semithick] (-2,0) -- node {\midarrow} ({-2+cos 60}, {sin 60});
\draw[semithick] ({-2+cos 60}, {sin 60}) -- node {\midarrow} ({-3+cos 60}, {sin 60});
\draw[semithick] ({-3+cos 60}, {sin 60}) -- node {\midarrow} (-2,0);

\draw[semithick] ({-2+cos 60}, {sin 60}) -- node {\midarrow} ({-2+2*cos 60}, {2*sin 60});
\draw[semithick] ({-2+2*cos 60}, {2*sin 60}) -- node {\midarrow} ({-3+2*cos 60}, {2*sin 60});
\draw[semithick] ({-3+2*cos 60}, {2*sin 60}) -- node {\midarrow} ({-2+cos 60},{sin 60});

\draw[semithick] ({-cos 60, sin 60}) -- node{\midarrow} ({cos 60}, {sin 60});
\draw[semithick] ({cos 60, sin 60}) -- node{\midarrow} (0, {2*sin 60});
\draw[semithick] (0, {2*sin 60}) -- node{\midarrow} ({-cos 60}, {sin 60});

\draw[semithick] (0,{2*sin 60}) -- node{\midarrow} (1, {2*sin 60});
\draw[semithick] (1, {2*sin 60}) -- node{\midarrow} ({cos 60},{3*sin 60});
\draw[semithick] ({cos 60, 3*sin 60}) -- node{\midarrow} (0, {2*sin 60});

\draw[semithick] (-1,{2*sin 60}) -- node{\midarrow} (0, {2*sin 60});
\draw[semithick] (0, {2*sin 60}) -- node{\midarrow} ({-cos 60},{3*sin 60});
\draw[semithick] ({-cos 60, 3*sin 60}) -- node{\midarrow} (-1, {2*sin 60});

\draw[semithick]  ({cos 60, -sin 60}) -- node{\midarrow} ({-cos 60}, {-sin 60});
\draw[semithick] ({-cos 60, -sin 60}) -- node{\midarrow} (0, {-2*sin 60});
\draw[semithick] (0, {-2*sin 60}) -- node{\midarrow} ({cos 60}, {-sin 60});

\draw[semithick] (0,{-2*sin 60}) -- node{\midarrow} (-1, {-2*sin 60});
\draw[semithick] (-1, {-2*sin 60}) -- node{\midarrow} ({-cos 60},{-3*sin 60});
\draw[semithick] ({-cos 60, -3*sin 60}) -- node{\midarrow} (0, {-2*sin 60});

\draw[semithick] (1,{-2*sin 60}) -- node{\midarrow} (0, {-2*sin 60});
\draw[semithick] (0, {-2*sin 60}) -- node{\midarrow} ({cos 60},{-3*sin 60});
\draw[semithick] ({cos 60, -3*sin 60}) -- node{\midarrow} (1, {-2*sin 60});

\draw[semithick] (1,0) -- node {\midarrow} ({1-cos 60}, {-sin 60});
\draw[semithick] ({1-cos 60}, {-sin 60}) -- node {\midarrow} ({2-cos 60}, {-sin 60});
\draw[semithick] ({2-cos 60}, {-sin 60}) -- node {\midarrow} (1,0);

\draw[semithick] (2,0) -- node {\midarrow} ({2-cos 60}, {-sin 60});
\draw[semithick] ({2-cos 60}, {-sin 60}) -- node {\midarrow} ({3-cos 60}, {-sin 60});
\draw[semithick] ({3-cos 60}, {-sin 60}) -- node {\midarrow} (2,0);

\draw[semithick] ({2-cos 60}, {-sin 60}) -- node {\midarrow} ({2-2*cos 60}, {-2*sin 60});
\draw[semithick] ({2-2*cos 60}, {-2*sin 60}) -- node {\midarrow} ({3-2*cos 60}, {-2*sin 60});
\draw[semithick] ({3-2*cos 60}, {-2*sin 60}) -- node {\midarrow} ({2-cos 60},{-sin 60});

\draw[semithick] (-1,0) -- node {\midarrow} ({-1-cos 60}, {-sin 60});
\draw[semithick] ({-1-cos 60}, {-sin 60}) -- node {\midarrow} ({-cos 60}, {-sin 60});
\draw[semithick] ({-cos 60}, {-sin 60}) -- node {\midarrow} (-1,0);

\draw[semithick] (-2,0) -- node {\midarrow} ({-2-cos 60}, {-sin 60});
\draw[semithick] ({-2-cos 60}, {-sin 60}) -- node {\midarrow} ({-1-cos 60}, {-sin 60});
\draw[semithick] ({-1-cos 60}, {-sin 60}) -- node {\midarrow} (-2,0);

\draw[semithick] ({-1-cos 60}, {-sin 60}) -- node {\midarrow} ({-1-2*cos 60}, {-2*sin 60});
\draw[semithick] ({-1-2*cos 60}, {-2*sin 60}) -- node {\midarrow} ({-2*cos 60}, {-2*sin 60});
\draw[semithick] ({-2*cos 60}, {-2*sin 60}) -- node {\midarrow} ({-1-cos 60},{-sin 60});

\draw[fill,blue] (0,{1.5*sin 60}) circle (0.05);
\draw[fill, blue] ({cos 60, 2.5*sin 60}) circle (0.05);
\draw[fill, blue] ({-cos 60, 2.5*sin 60}) circle (0.05);
\draw[semithick, blue] (0,{1.5*sin 60}) -- (0,{2*sin 60});
\draw[semithick, blue] ({cos 60,2.5*sin 60}) -- (0,{2*sin 60});
\draw[semithick, blue] ({-cos 60,2.5*sin 60}) -- (0,{2*sin 60});

\draw[fill,blue] (1,{0.5*sin 60}) circle (0.05);
\draw[fill,blue] (2,{0.5*sin 60}) circle (0.05);
\draw[fill,blue] (1.5,{1.5*sin 60}) circle (0.05);

\draw[semithick, blue] (1,{0.5*sin 60}) -- (1.5,{sin 60});
\draw[semithick, blue] ({2, 0.5*sin 60}) -- (1.5,{sin 60});
\draw[semithick, blue] ({1.5,1.5*sin 60}) -- (1.5,{sin 60});

\draw[fill,blue] (-1,{0.5*sin 60}) circle (0.05);
\draw[fill,blue] (-2,{0.5*sin 60}) circle (0.05);
\draw[fill,blue] (-1.5,{1.5*sin 60}) circle (0.05);

\draw[semithick, blue] (-1,{0.5*sin 60}) -- (-1.5,{sin 60});
\draw[semithick, blue] ({-2, 0.5*sin 60}) -- (-1.5,{sin 60});
\draw[semithick, blue] ({-1.5,1.5*sin 60}) -- (-1.5,{sin 60});

\draw[fill,blue] (0,{-1.5*sin 60}) circle (0.05);
\draw[fill, blue] ({cos 60, -2.5*sin 60}) circle (0.05);
\draw[fill, blue] ({-cos 60, -2.5*sin 60}) circle (0.05);
\draw[semithick, blue] (0,{-1.5*sin 60}) -- (0,{-2*sin 60});
\draw[semithick, blue] ({cos 60,-2.5*sin 60}) -- (0,{-2*sin 60});
\draw[semithick, blue] ({-cos 60,-2.5*sin 60}) -- (0,{-2*sin 60});

\draw[fill,blue] (-1,{-0.5*sin 60}) circle (0.05);
\draw[fill,blue] (-2,{-0.5*sin 60}) circle (0.05);
\draw[fill,blue] (-1.5,{-1.5*sin 60}) circle (0.05);

\draw[semithick, blue] (-1,{-0.5*sin 60}) -- (-1.5,{-sin 60});
\draw[semithick, blue] ({-2, -0.5*sin 60}) -- (-1.5,{-sin 60});
\draw[semithick, blue] ({-1.5,-1.5*sin 60}) -- (-1.5,{-sin 60});

\draw[fill,blue] (1,{-0.5*sin 60}) circle (0.05);
\draw[fill,blue] (2,{-0.5*sin 60}) circle (0.05);
\draw[fill,blue] (1.5,{-1.5*sin 60}) circle (0.05);

\draw[semithick, blue] (1,{-0.5*sin 60}) -- (1.5,{-sin 60});
\draw[semithick, blue] ({2, -0.5*sin 60}) -- (1.5,{-sin 60});
\draw[semithick, blue] ({1.5,-1.5*sin 60}) -- (1.5,{-sin 60});

\draw[semithick,blue] (0,{1.5*sin 60}) -- (1, {0.5*sin 60});
\draw[semithick,blue] (1,{0.5*sin 60}) -- (1, {-0.5*sin 60});
\draw[semithick,blue] (1,{-0.5*sin 60}) -- (0, {-1.5*sin 60});
\draw[semithick,blue] (0,{-1.5*sin 60}) -- (-1, {-.5*sin 60});
\draw[semithick,blue] (-1,{-0.5*sin 60}) -- (-1, {.5*sin 60});
\draw[semithick,blue] (-1,{0.5*sin 60}) -- (0, {1.5*sin 60});

\draw[semithick,blue] (0.5, {2.5*sin 60}) -- (1.5, {1.5*sin 60});
\draw[semithick,blue] (2, {.5*sin 60}) -- (2, {-.5*sin 60});
\draw[semithick,blue] (1.5, {-1.5*sin 60}) -- (0.5, {-2.5*sin 60});

\draw[semithick,blue] (-0.5, {2.5*sin 60}) -- (-1.5, {1.5*sin 60});
\draw[semithick,blue] (-2, {.5*sin 60}) -- (-2, {-.5*sin 60});
\draw[semithick,blue] (-1.5, {-1.5*sin 60}) -- (-0.5, {-2.5*sin 60});

\draw[semithick,blue] (0.5,{2.5*sin 60}) -- (0.5,{3.25*sin 60});
\draw[semithick,blue] (-0.5,{2.5*sin 60}) -- (-0.5,{3.25*sin 60});
\draw[semithick,blue] (0.5,{-2.5*sin 60}) -- (0.5,{-3.25*sin 60});
\draw[semithick,blue] (-0.5,{-2.5*sin 60}) -- (-0.5,{-3.25*sin 60});
\draw[semithick,blue] (1.5,{1.5*sin 60}) -- ({2.25, 2.25*sin 60});
\draw[semithick,blue] (-1.5,{1.5*sin 60}) -- ({-2.25, 2.25*sin 60});
\draw[semithick,blue] (2,{0.5*sin 60}) -- ({2.75, 1.25*sin 60});
\draw[semithick,blue] (-2,{0.5*sin 60}) -- ({-2.75, 1.25*sin 60});

\draw[semithick,blue] (1.5,{-1.5*sin 60}) -- ({2.25, -2.25*sin 60});
\draw[semithick,blue] (-1.5,{-1.5*sin 60}) -- ({-2.25, -2.25*sin 60});
\draw[semithick,blue] (2,{-0.5*sin 60}) -- ({2.75, -1.25*sin 60});
\draw[semithick,blue] (-2,{-0.5*sin 60}) -- ({-2.75, -1.25*sin 60});

\end{scope}

\end{tikzpicture}
\end{center}
\caption{The dual Lagrangian cellulation in a hexagon of ideal triangles for $\Delta_2$\label{Fig:dual_lagrangians}}
\end{figure}

The vertices of this dual `Lagrangian' cellulation $\Delta_m^{\vee}$ are at the centres of the inscribed black triangles of $\Delta_m$;  there are $m(m+1)/2$ vertices of $\Delta_m^{\vee}$ in each ideal triangle of $\Delta$.  Thus, the total number of vertices of the Lagrangian cellulation is
\[
m(m+1)/2 \cdot (4g-4+2d) = m(m+1)(2g-2+d)
\]
which co-incides with the number of branch points of $\Sigma \to S$, cf. Remark \ref{rmk:numerics}.

\begin{lem}\label{lem:delta_m_controls_singular_fibres}
Given an ideal triangulation $\Delta$ of $(\bS,\bP)$, there is a tuple $\Phi$ such that, up to isotopy, the associated fibration $Y_{\Phi} \to S$ has reducible fibres at the points of $\bP$, and has Lefschetz singular fibres over the vertices of $\Delta_m^{\vee}$.
\end{lem}

\begin{proof} 
Recall from \cite{GMN, BridgelandSmith} that a generic choice of quadratic differential $\phi_2$ on $S$ with double poles at $D\subset S$ defines an ideal triangulation $\Delta$, with a (simple) zero of $\phi_2$ at the centre of each triangle.  We consider a point of the higher rank Hitchin base $(\phi_2,\ldots, \phi_m)$, with $\phi_j \in H^0(K_S(D)^{\otimes j})$, which is a small perturbation of the degenerate tuple $(\phi_2,0,\beta_1\cdot \phi_2^2,0,\beta_2\cdot \phi_2^4,\ldots)$ for constants $\beta_j \in \bC^*$ chosen so that the associated spectral curve factorizes:
\begin{equation}\label{eqn:reducible_spectral_curve}
\begin{aligned}
\Sigma_0 & = & \{(b^2 - \phi_2)(b^2-\alpha_1\phi_2)\cdots(b^2-\alpha_k\phi_2) = 0\} \subset \mathrm{Tot}(K_S(D)) \qquad m=2k+1; \\ 
\Sigma_0 & = & \{b \,(b^2-\phi_2)(b^2-\alpha_1\phi_2)\cdots(b^2-\alpha_k\phi_2) = 0\} \subset \mathrm{Tot}(K_S(D)) \qquad m=2k+2;
\end{aligned}
\end{equation}
where $b$ denotes a co-ordinate on $K_S(D)$ and the $\alpha_j$ are pairwise distinct and not equal to $0$ or $1$. These curves are reducible, and the 3-fold conic fibration over $K_S(D)$ with discriminant $\Sigma_0$ has (for sufficiently general $\alpha_j$) an isolated singularity at each zero of $\phi_2$. In local co-ordinates  near a zero of $\phi_2 \approx z$ on $S$ (and recalling the zeroes of $\delta$ and of $\phi_2$ are different so locally $\delta \approx 1$), the local model  for the 3-fold is
\[
\{ac = (b^2-z)\cdots(b^2-\alpha_k z)\} \subset \bC^4 \qquad \mathrm{or} \qquad \{ac = b(b^2-z)\cdots(b^2-\alpha_k z)\} \subset \bC^4,
\]
which is the stabilisation of a weighted homogeneous plane curve singularity with Milnor number $m(m-1)/2$.  A small perturbation of the degenerate tuple to a generic tuple $\Phi$ of differentials will both smooth the threefold, and yield a smooth spectral curve which is generically branched over the $z$-plane, the branch points encoding the positions of the Lefschetz singularities of $Y_\Phi \to S$. 
(The affine modifications which relate $Y_\Phi$ to a conic bundle singular along the spectral curve do not affect the current local discussion, since they take place in fibres over $D$, which lie far from the singularities of the total space for the degenerate tuple, and from the locations of the Lefschetz singularities of the fibration after small perturbation of that tuple.)

The topology of the threefold is thus encoded, up to birational modifications far from the Lefschetz singularities, by the braid monodromy of the smoothed spectral curve $\Sigma \subset K_S(D)$.  Away from points of $D \subset S$, the initial 3-fold has fibre locally modelled on $\{ac = b^{m+1}\}$ near an isolated simple zero of the differential $\phi_2$; the singularity in the total space is an $A_m$-singularity in the fibre.  A small generic perturbation of the tuple gives rise to a smooth 3-fold locally cut out by $\{ac=b^{m+1} + P(b,z)\}$ in which  the corresponding map from the $z$-plane to configurations of roots is transverse to the discriminant locus of repeated roots; the Lefschetz singular fibres of the projection $Y\to S$ local to the given zero of $\phi_2$ then arise from values $z$ where $b^{m+1} + P(b,z)$ has a double root. The discriminant of \eqref{eqn:reducible_spectral_curve} has degree $(2k+2)(2k+1)/2$ respectively $(2k+3)(2k+2)/2$ (as a function of $z$) in the two cases, which in both cases yields the value $m(m+1)/2$. After a smooth area-preserving isotopy of the base, which can be lifted to a symplectic isotopy of the total space, one can arrange that these Lefschetz critical points lie at the vertices of $\Delta_m^{\vee}$.  Compare to \cite[Figures 1--3]{GMN:snakes}, and Proposition \ref{prop:sphere_configuration} below. 

Over points of $\bP$, i.e. points of the divisor $D$ where $\delta$ vanishes (simply), a local model for the 3-fold is given by $\{ \delta \,ac = b^{m+1}-1\} \subset \bC^4$, and the $\delta=0$ fibre is given by $m+1$ pairwise disjoint copies of $\bC^2$ with co-ordinates $(u,v)$.
\end{proof}

\begin{rmk} \label{rmk:exceptional} The case $g=1$ and $|\bP| = 1$ of an elliptic curve $S=E$ with one marked point $D=\{p\}$  is exceptional; in that case $\delta$ and $b$ both belong to the same one-dimensional space $H^0(\mathcal{O}_E(p))$, and the equation for the threefold associated to a reducible spectral curve becomes degenerate.  However, after including the perturbation terms the corresponding threefold is still smooth. \end{rmk}

\begin{rmk}
We do not assert that there is a tuple $\Phi$ for which the holomorphic projection has the described structure, only that it is symplectically isotopic to such.  
\end{rmk}

\begin{lem} \label{lem:H2Y}
The rational cohomology $H^2(Y_{\Phi};\bQ)$ has rank $dm+1$, and is spanned by the components of the reducible fibres over points of $\bP \subset \bS$, modulo the relation that their sum is independent of $p\in \bP$.
\end{lem}

\begin{proof} The rank computation is given in \cite[Section 6]{Abrikosov}. The generators can be extracted from his argument.  Note that the total fibre class, which co-incides with the sum of the classes of the components of a fixed reducible fibre, agrees with the pullback of an area form on $\bS$ of total area $1$.
\end{proof}

When $m=2$, the space of K\"ahler forms on $Y_{\Phi}$ is an open cone of dimension $2d+1$, whilst the space of right equivalence classes of potentials has dimension $d+2$ by Remark \ref{rmk:abrikosov}. This underscores the fact that one cannot expect the `mirror' map in \eqref{eqn:project_potential} to be a local isomorphism. 

\begin{rmk}\label{rmk:kahler_monodromy}
There is a family of 3-folds $Y_{\Phi}$ over the space of generic tuples $\Phi$, and it is natural to consider K\"ahler forms which vary locally trivially over the family. The monodromy permutes components of the reducible fibres, and the monodromy-invariant subspace of $H^2(Y_{\Phi};\bR)$ has rank 2.  Up to global rescaling,  there is thus just a one-parameter family of invariant K\"ahler forms.\end{rmk}

\subsection{Tripod spheres\label{Sec:Tripods}}

We recall Donaldson's `matching sphere' construction.  Consider a symplectic Lefschetz fibration $X^{2n} \to \bC$ with two singular fibres lying over $\pm 1$, and a path $\gamma: [-1,1] \to \bC$ with $\gamma(\pm 1) = \pm 1$ and $\gamma(t) \not \in \{-1,+1\}$ for $t\in (-1,1)$.  Parallel transport along $\gamma$ gives rise to two Lagrangian $S^{n-1}$ vanishing cycles in the fibre $X_{\gamma(0)}$. If these are Hamiltonian isotopic, then after a deformation of the symplectic connexion on $X$ in a neighbourhood of the preimage of $\gamma$, one can arrange that the vanishing cycles agree exactly, and glue to form a Lagrangian $S^n \subset X$, which is well-defined up to Hamiltonian isotopy.  In this case, $\gamma$ is called a `matching path'; see \cite[Lemma 8.4]{AMP} and \cite[Section 16g]{Seidel:FCPLT}  for the details of the  construction. Because of the deformation of symplectic connexions, the matching sphere will in general only approximately lie over $\gamma$ for the original symplectic form $\omega$.  

If $X \to B$ is a symplectic Lefschetz fibration with fibre $T^*S^2 = A_1$, then since the fibre contains a unique Lagrangian sphere up to Hamiltonian isotopy \cite{Hind}, any path between critical values in the base  (disjoint from critical values in its interior) is a matching path. For $A_m$ fibred 3-folds with $m>1$, the fibre contains infinitely many Hamiltonian isotopy classes of Lagrangian sphere.  It will be useful to have a minor generalisation of the matching path construction, namely a  `matching tripod' construction giving rise to `tripod' Lagrangian spheres.

The $A_m$-surface $\{x^2+y^2+\prod_{j=1}^{m+1} (z-j) = 0\} \subset \bC^3$ inherits an exact  K\"ahler structure from $(\bC^3, \omega_{\st})$. It deformation retracts to a compact core (or skeleton) comprising an $A_m$-chain of Lagrangian spheres, which arise as matching spheres for the paths $[j.j+1] \subset\bR\subset\bC$ for $1 \leq j \leq m$.  Let $a$ and $b$ denote a pair of Lagrangian 2-spheres in $A_m$ which meet transversely at a single point, for instance (but not necessarily) a consecutive pair of spheres in the compact core, see Figure \ref{Fig:matching_paths_in_A2}.

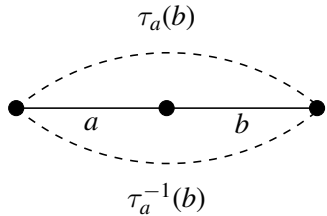
\begin{figure}[ht]
\begin{center}
\begin{tikzpicture}

\draw[semithick] (-2,0) -- (0,0) -- (2,0);
\draw[fill] (-2,0) circle (0.1);
\draw[fill] (0,0) circle (0.1);
\draw[fill] (2,0) circle (0.1);
\draw[semithick,dashed] (-2,0) arc (130:50:3.15);
\draw[semithick,dashed] (-2,0) arc (230:310:3.15);
\draw (-1,-0.2) node {$a$};
\draw (1,-0.2) node {$b$};
\draw (0,1.2) node {$\tau_a(b)$};
\draw (0,-1.2) node {$\tau_a^{-1}(b)$};

\end{tikzpicture}
\end{center}
\caption{Matching paths in the $A_2$-surface\label{Fig:matching_paths_in_A2}}
\end{figure}

\begin{lem} \label{lem:tripod_vs_matching_sphere}
Let $p: X \to \bC$ be a symplectic Lefschetz fibration with three singular fibres and fibre $A_m$ with $m>1$. Suppose that the vanishing cycles are as shown in the first image of Figure \ref{Fig:tripod_paths}, i.e. $a, \tau_a^{-1}(b), b$ for paths as drawn. Then there is a Lagrangian sphere which maps under $p$ to a small neighbourhood of the tripod spanning the three critical points.
\end{lem}

\begin{figure}[ht]
\begin{center}
\begin{tikzpicture}[scale = 0.6]

\draw[fill] (0,-1.7) circle (0.1);
 \draw[fill] (1.5,1.5) circle (0.1);
 \draw[fill] (-1.5,1.5) circle (0.1);

\draw[semithick, rounded corners] (0,-1.7) -- (0,-0.5) -- (0.3,0.3) -- (1.5,1.5); 
\draw[semithick, rounded corners] (0,-1.7) -- (0,-0.5) -- (-0.3,0.3) -- (-1.5,1.5); 
\draw[semithick, rounded corners] (-1.5,1.5) -- (-0.3,0.3) -- (0.3,0.3) --(1.5,1.5);
\draw (0.8,0) node {$L$};

\draw[fill] (-6+0,-1.7) circle (0.1);
 \draw[fill] (-6+1.5,1.5) circle (0.1);
 \draw[fill] (-6+-1.5,1.5) circle (0.1);
\draw (-6+1.5,-3) circle (0.1);
\draw[semithick] (-4.5,-3) -- (-6,-1.7);
\draw[semithick] (-4.5,-3) arc (10:56:6.8);
\draw[semithick] (-4.5,-3) arc (-22:22:6);

\draw (-4.4,-3.5) node {$q$};
\draw (-5.5,-2.6) node {$a$};
\draw (-7,0) node {$\tau_a^{-1}(b)$};
\draw (-3.6,0) node {$b$};

\draw[fill] (+6+0,-1.7) circle (0.1);
 \draw[fill] (+6+1.5,1.5) circle (0.1);
 \draw[fill] (+6+-1.5,1.5) circle (0.1);
\draw (+6+1.5,-3) circle (0.1);
\draw[semithick, dashed] (12-4.5,-3) -- (12-6,-1.7);
\draw[semithick, dashed] (12-4.5,-3) arc (10:56:6.8);
\draw[semithick] (12-4.5,-3) arc (-22:22:6);
\draw[semithick] (12-4.5,-3) arc (280:150:3);

\draw (12-4.4,-3.5) node {$q$};
\draw (12-5.5,-2.6) node {$a$};
\draw (12-3.6,0) node {$b$};
\draw (12-7.5, -2.6) node {$b$};

\end{tikzpicture}
\end{center}
\caption{Matching paths for the Lagrangian tripod (left); the tripod sphere (middle); the tripod sphere as a matching sphere (right)\label{Fig:tripod_paths}}
\end{figure}
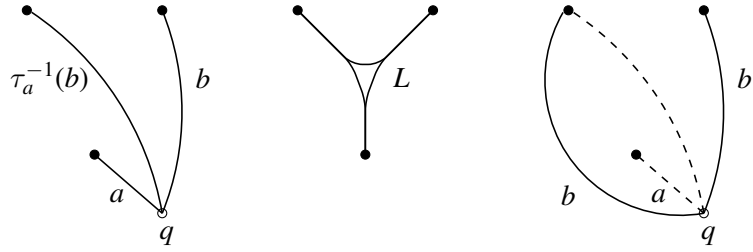

\begin{proof} 
The existence of a Lagrangian 3-sphere  $L$ mapping to the tripod neighbourhood follows from the construction of a Lagrangian cobordism from a Lagrange surgery. We apply this to the Polterovich surgery of the core spheres $a$ and $b$ in the $A_2$-Milnor fibre; this yields a cobordism in the total space of the product of an $A_2$-surface and a disc, whose three ends carry the three Lagrangians $a,b,\tau_a^{-1}(b)$.  The cobordism completes to a smooth closed Lagrangian submanifold of a Lefschetz fibration with three critical points as in the left picture, compare to \cite{BiranCornea:lefschetz}. It is straightforward to check that the resulting closed Lagrangian is a smooth sphere.

On the right side of Figure \ref{Fig:tripod_paths},  the two outer paths both have vanishing cycle $b$, so  their concatenation defines a matching path $\gamma$ and Lagrangian 3-sphere $L_{\gamma}$  in the total space of the Lefschetz fibration. The fact that $L_{\gamma}$ is Hamiltonian isotopic to the tripod sphere $L$ constructed previously is proved by isotoping the matching path $\gamma$  through the `lowest' critical value, see \cite[Lemma A.25]{AbouzaidSmith19} for details. (Cancelling one critical point and one handle in the fibre by a Weinstein surgery as in \cite{Cieliebak-Eliashberg},  the total space of the Lefschetz fibration $X$ is symplectomorphic to $T^*S^3$, with the matching sphere for $\gamma$ and hence the tripod sphere $L$ Hamiltonian isotopic to the zero-section.)
 \end{proof}

The tripod sphere does not map exactly to a tripod of arcs, but to a neighbourhood of that which is fattened near the vertex; we will call such spheres `essentially fibred'.

\begin{rmk}
If we had taken the matching paths $\langle b,\tau_a(b),a \rangle$ as the input ordered triple (for the same vanishing paths), rather than $\langle b, \tau_a^{-1}(b), a\rangle$, then we would still obtain a Lagrangian sphere in the total space, but the corresponding description as a matching sphere would break, since the leftmost vanishing path on the right image of Figure \ref{Fig:tripod_paths} would have associated vanishing cycle $\tau_a^2(b) \not\simeq b$.  (More precisely, the matching sphere would now lie over a path in a different homotopy class on the right hand picture.) \end{rmk}

\begin{rmk} The right hand picture of Figure \ref{Fig:tripod_paths} is symmetric in a way which is not manifest.  See Figure \ref{Fig:more_tripod_paths}.  There are matching spheres over both non-dashed paths (or the analogous 3rd path, not shown), using the fact that $\tau_a^{-1}(b) \simeq \tau_b(a)$ in the $A_2$-fibre, compare to Figure \ref{Fig:matching_paths_in_A2}. \end{rmk}

\begin{figure}[ht]
\begin{center}
\begin{tikzpicture}[scale=0.7]

\draw[fill] (+0,-1.7) circle (0.1);
 \draw[fill] (+1.5,1.5) circle (0.1);
 \draw[fill] (+-1.5,1.5) circle (0.1);
\draw (+1.5,-3) circle (0.1);
\draw (-1.5,-2.5) circle (0.1);
\draw[semithick] (-1.5,-2.5) arc (270:330:1.7);
\draw[semithick] (-1.5,-2.5) arc (240:45:2.55);

\draw (-3,1) node {$ a$};
\draw (-1,-2) node {$a$};

\draw[semithick, dashed] (6-4.5,-3) -- (6-6,-1.7);
\draw[semithick, dashed] (6-4.5,-3) arc (10:56:6.8);
\draw (6-5.5,-2.6) node {$a$};
\draw (6-5.7,0.7) node {$\tau_b(a)$};
\draw[semithick, dashed] (6-4.5,-3) arc (-22:22:6);
\draw (6-3.6,0) node {$b$};

\draw[fill] (+6+0,-1.7) circle (0.1);
 \draw[fill] (+6+1.5,1.5) circle (0.1);
 \draw[fill] (+6+-1.5,1.5) circle (0.1);
\draw (+6+1.5,-3) circle (0.1);
\draw[semithick, dashed] (12-4.5,-3) -- (12-6,-1.7);
\draw[semithick, dashed] (12-4.5,-3) arc (10:56:6.8);
\draw[semithick] (12-4.5,-3) arc (-22:22:6);
\draw[semithick] (12-4.5,-3) arc (280:150:3);

\draw (12-4.4,-3.5) node {$q$};
\draw (12-5.5,-2.6) node {$a$};
\draw (12-5.5,0.7) node {$\tau_a^{-1}(b)$};
\draw (12-3.6,0) node {$b$};
\draw (12-7.5, -2.6) node {$b$};

\draw (6-4.4,-3.5) node {$q$};
\draw (-1.5,-3) node {$q'$};

\end{tikzpicture}
\end{center}
\caption{Symmetry of matching spheres from tripod spheres: in $A_2$,  $\tau_a^{-1}(b) = \tau_b(a)$ and $\tau_{\tau_b(a)}(b) = a$\label{Fig:more_tripod_paths}}
\end{figure}
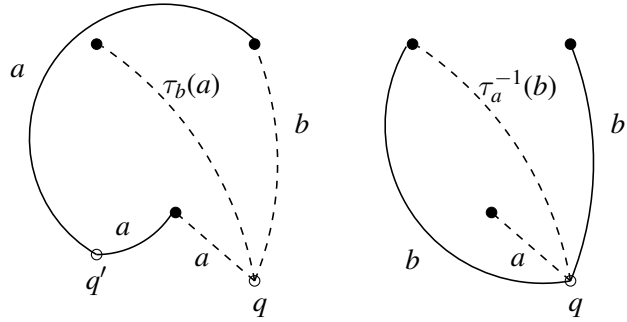

The essential point of the symmetry of the `good' tripod on the right of Figure \ref{Fig:tripod_paths} is that one can arrange a collection of tripods around a polygon so that the corresponding vanishing cycles agree and the configuration `closes up', cf. Figure \ref{Fig:tripod_cycle}.

\begin{figure}[ht]
\begin{center}
\begin{tikzpicture}[scale =1]

\draw[fill] (-1.25,0) circle (0.05);
\draw[fill] (1.25,0) circle (0.05);
\draw[fill] (-2,-2) circle (0.05);
\draw[fill] (2,-2) circle (0.05);
\draw[fill] (0,-3.5) circle (0.05);
\draw (2,0.5) circle (0.05);
\draw (-2,0.5) circle (0.05);
\draw[fill] (2.75,0.5) circle (0.05);
\draw[fill] (-2.75,0.5) circle (0.05);
\draw[fill] (2,1.25) circle (0.05);
\draw[fill] (-2,1.25) circle (0.05);

\draw (2.75,-2) circle (0.05);
\draw[fill] (3.25,-2.5) circle (0.05);
\draw[fill] (3.25,-1.5) circle (0.05);
\draw (-2.75,-2) circle (0.05);
\draw[fill] (-3.25,-2.5) circle (0.05);
\draw[fill] (-3.25,-1.5) circle (0.05);

\draw (0,-4.25) circle (0.05);
\draw[fill] (-.5,-4.75) circle (0.05);
\draw[fill] (0.5,-4.75) circle (0.05);

\draw (0,-1.75) circle (0.05);

\draw[semithick] (-1.25,0) -- (1.25,0) -- (2,-2) -- (0,-3.5) -- (-2,-2) -- cycle;
\draw[semithick] (1.25,0) -- (2, 0.5);
\draw[semithick] (2,0.5) -- (2.75, 0.5);
\draw[semithick] (2,0.5) -- (2,1.25);
\draw[semithick] (-1.25,0) -- (-2, 0.5);
\draw[semithick] (-2,0.5) -- (-2.75, 0.5);
\draw[semithick] (-2,0.5) -- (-2,1.25);

\draw[semithick] (2,-2) -- (2.75,-2);
\draw[semithick] (2.75,-2) -- (3.25,-1.5);
\draw[semithick] (2.75,-2) -- (3.25,-2.5);
\draw[semithick] (-2,-2) -- (-2.75,-2);
\draw[semithick] (-2.75,-2) -- (-3.25,-1.5);
\draw[semithick] (-2.75,-2) -- (-3.25,-2.5);

\draw[semithick] (0,-4.25) -- (0,-3.5);
\draw[semithick] (0,-4.25) -- (-.5,-4.75);
\draw[semithick] (0,-4.25) -- (.5,-4.75);

\draw[semithick, dashed] (2.75,0.5) -- (3.25,-1.5);
\draw[semithick, dashed] (-2.75,0.5) -- (-3.25,-1.5);
\draw[semithick, dashed] (.5,-4.75) -- (3.25,-2.5);
\draw[semithick, dashed] (-.5,-4.75) -- (-3.25,-2.5);
\draw[semithick, dashed] (2,1.25) -- (-2,1.25);

\draw (0,-.25) node {$a$};
\draw (1.25,-1) node {$a$};
\draw (-1.25,-1) node {$a$};
\draw (-1,-2.5) node {$a$};
\draw (1,-2.5) node {$a$};
\draw (-1.5,0.35) node {$a$};
\draw (-1.85,0.9) node {$b$};
\draw (0,1) node {$b$};
\draw (1.75,-3.3) node {$b$};
\draw (2.4,-1.8) node {$a$};

\draw[semithick, dotted, rounded corners, red] (0,-1.75) -- (1,0.75) -- (2,1.25);
\draw[semithick, dotted, rounded corners, red] (0,-1.75) -- (2,-1) -- (2.75,0.5);
\draw[semithick, dotted, rounded corners,red] (0,-1.75) -- (2,-2.5) -- (3.25,-2.5);
\draw[semithick, dotted, rounded corners, red] (0,-1.75) -- (1.2,0.6) -- (2.2,0.8) -- (2.75,0.5);

\draw[red] (1, 0.75) node {$b$};
\draw[red] (2.4,-0.9) node {$b$};
\draw[red] (2,-2.5) node {$b$};
\draw[red] (2.6,0.95) node {$\tau_a^{-1}(b)$};

\end{tikzpicture}
\end{center}
\caption{A cycle of tripod Lagrangians; labels denote vanishing cycles associated to dashed /  radial  paths \label{Fig:tripod_cycle}}
\end{figure}
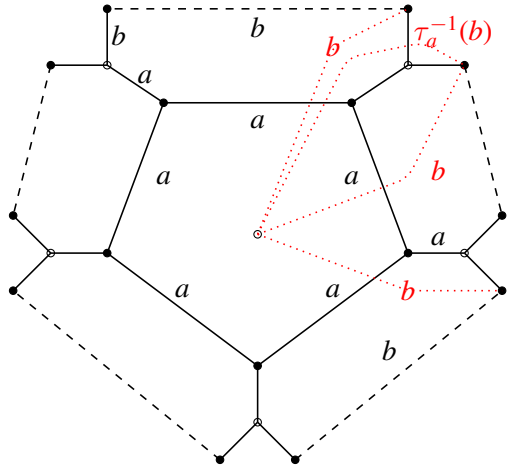

The arrangement displayed in Figure \ref{Fig:tripod_cycle} is local and topological: given two Lagrangian spheres $L_a, L_b \subset A_m$ which meet transversely once, one can construct a Lefschetz fibration over a disc with the given critical fibres and vanishing cycles. We need a globalisation of this local picture. 

\subsection{Spectral networks and sphere configurations}

Recall that the threefold $Y_{\Phi}$ is defined by an equation $\delta a c = b^{m+1} - \sum_{j=2}^{m+1} \phi_j b^{m+1-j} = 0$ in the total space of a vector bundle $\mathcal{W}$ over $S$.  Because we have well-defined parallel transport over paths in $S$, we can consider matching and tripod spheres in the total space. 

\begin{prop} \label{prop:sphere_configuration}
Fix an ideal triangulation $\Delta$ and the dual Lagrangian cellulation $\Delta_m^{\vee}$ of its subdivision $\Delta_m$. There is a configuration of Lagrangian spheres essentially fibred over $\Delta_m^{\vee}$, in which each ideal triangle in $\Delta$ contains $m(m-1)/2$ tripod Lagrangians, and these clusters are joined by $m$ matching spheres for each edge of $\Delta$.
\end{prop}

\begin{proof} 
This is an extension of Lemma \ref{lem:delta_m_controls_singular_fibres}, and is again implicit in the work of Gaiotto-Moore-Neitzke relating their spectral networks to ideal triangulations \cite{GMN:snakes}.  Focus first on the geometry inside a single ideal triangle. In the $A_1$ situation, the spectral cover is a double cover and the monodromy at a simple zero of $\phi_2$ swaps the sheets. We are taking a perturbation of a degenerate case in which we replace $b^2-\phi$ by \eqref{eqn:reducible_spectral_curve}, for which the monodromy around a zero of $\phi_2$ completely reverses the order of the sheets, giving the longest element $(1, m) (2, m-1) (3, m-2) \cdots$ of the symmetric group. The braid monodromy for such a reducible curve was computed in \cite[Section 5]{Cohen-Suciu} \cite[Lemma 4.1]{VietDung}, and yields 
the Garside element of the braid group. This is the lift of the longest element of the symmetric group, which admits a factorization as the canonical sequence of $m(m+1)/2$  half-twists lifting the permutations \begin{multline}
[(m, m-1) (m-1, m-2)\ldots  (4, 3) (3, 2) (2, 1)] \cdot  [(m, m-1) (m-1, m-2) \ldots(4, 3) (3, 2)] \cdot \\ \cdot  [(m, m-1)\ldots (4, 3)]\cdot  \ldots \cdot [(m, m-1) (m-1, m-2)] \cdot [(m, m-1)],
\end{multline}
compare to \cite{Brieskorn-Saito}, \cite[Section 2]{Looijenga} and the labellings of the pairs of sheets  in Figure \ref{Fig:branching_data}. The above factorisation \emph{defines} a local smooth symplectic surface in the four-ball simply branched over the $z$-plane, compare to \cite{Loi-Piergallini,Orevkov}, and gives a local model for the smoothing of the spectral curve in  the proof of Lemma \ref{lem:delta_m_controls_singular_fibres}.

In any ideal triangle for $\Delta_2^{\vee}$, three sheets of the spectral cover interact. The vertices of $\Delta_m^{\vee}$ can be grouped into (overlapping) triples governed by the same local geometry, see Figure \ref{Fig:branching_data} (and compare to the corresponding discussion around \cite[Figure 1]{GMN:snakes}).  Each such triple then bounds a Lagrangian tripod sphere, and the previous discussion implies these tripods fit together as in (the perhaps higher rank analogue of) Figure \ref{Fig:a3_lag_cellulation}.  

Before perturbation, the only branching of the reducible spectral curve happens at the zeroes of $\phi_2$ or equivalently the centres of the ideal triangles. These have fibrewise $A_m$-singularities in which the whole compact core of the fibre degenerates into the critical point. For any given edge between two ideal triangles, one can place the branch cuts away from that edge, so the labelling of sheets is consistent along the paths of the cellulation which cross the edge of the given ideal triangle; indeed, for a given point  $p \in D$ and associated vertex of the ideal triangulation, one can place all branch cuts across the edges of triangles which are not adjacent to $p$. (This labelling of sheets is incorporated into the data of the `eigen-ordering' at $p$ introduced below in Definition \ref{defn:eigen}.)   This yields a system of matching paths  across all edges of triangles adjacent to $p$ after perturbation; compare to Figure \ref{Fig:tripod_cycle}, cf. also the discussion in  \cite[Section 4]{GMN:snakes} of the `asymptotic behaviour of the $\mathcal{S}$-walls for lifted theories'.   

The local model near the reducible fibres from the end of Lemma \ref{lem:delta_m_controls_singular_fibres} is $\bC^*$-equivariant for a $\bC^*$-action of weight one in $\delta$. The local K\"ahler form can be deformed to be $S^1$-invariant, and the monodromy around the reducible fibre over a point  $p \in D$ is then symplectically trivial on the $A_m$-surface. The monodromy around the outer boundary of the  configuration of ideal  triangles adjacent to $p$ is a power of the Garside element, and acts either trivially on the compact core of the $A_m$-surface or preserving the core but  reversing the chosen order of its components, depending on the parity of the valence of $p$ in the ideal triangulation.   In either case, the configuration of matching spheres  constructed from the viewpoint of $p$ is compatible with that one would construct around another point $q\in D$.  \end{proof}

\begin{rmk} Consider the case $m=2$, $g=1$ and $|\bP|=1$, see Figure \ref{Fig:globalising_configurations}. As indicated by the labelled vanishing cycles, the monodromy around the boundary of a fundamental domain of the torus is the square $\Delta^2 = (\tau_b\tau_a\tau_b)^2$ of the Garside element, which is central in the braid group.  There is non-trivial monodromy around both generating loops\footnote{These loops do not have canonical lifts to elements of the fundamental group of the smooth locus.} for $\pi_1(T^2)$, i.e. the meridian and longitude depicted as the black boundaries of the fundamental domain; indeed, the underlying $m=1$ theory in this case has a single branch cut along each side, and the monodromy on both edges of the fundamental domain induces the permutation $(1,3)(2)$ of the three sheets of the spectral curve, compare to \cite[Figure 10]{Hollands-Neitzke}.  For a global K\"ahler form, the monodromy around the reducible fibre is non-trivial but centralises the braid group. The $\bC^*$-invariant model in Proposition \ref{prop:sphere_configuration} trivialises this monodromy by an isotopy which is not compactly supported at infinity. \end{rmk}

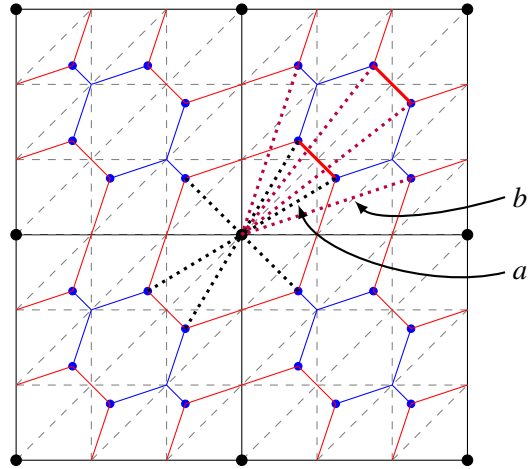
\begin{figure}[ht]
\begin{center}
\begin{tikzpicture}[scale=1]

\draw  (0,0) -- (3,0);
\draw  (0,0) -- (0,3);
\draw  (0,3) -- (3,3);
\draw  (3,0) -- (3,3);

\draw  (0,0) -- (-3,0);
\draw (-3,0) -- (-3,3);
\draw (0,3) -- (-3,3);

\draw (0,0) -- (0,-3);
\draw (0,-3) -- (3,-3);
\draw (3,0) -- (3,-3);

\draw (0,-3) -- (-3,-3);
\draw (-3,0) -- (-3,-3);

\draw[dashed,gray] (-3,-3) -- (3,3);
\draw[dashed,gray] (-3,-2) -- (2,3);
\draw[dashed,gray] (-3,-1) -- (1,3);
\draw[dashed,gray] (-3,0) -- (3,0);
\draw[dashed,gray] (-3,1) -- (-1,3);
\draw[dashed,gray] (-3,2) -- (-2,3);
\draw[dashed,gray] (-1,-3) -- (3,1);
\draw[dashed,gray] (1,-3) -- (3,-1);
\draw[dashed,gray] (0,-3) -- (3,0);
\draw[dashed,gray] (2,-3) -- (3,-2);
\draw[dashed,gray] (-3,0) -- (0,3);
\draw[dashed,gray] (-2,-3) -- (3,2);

\draw[dashed,gray] (-3,1) -- (3,1);
\draw[dashed,gray] (-3,2) -- (3,2);
\draw[dashed,gray] (-3,-1) -- (3,-1);
\draw[dashed,gray] (-3,-2) -- (3,-2);
\draw[dashed,gray] (-2,3) -- (-2,-3);
\draw[dashed,gray] (-1,3) -- (-1,-3);
\draw[dashed,gray] (1,3) -- (1,-3);
\draw[dashed,gray] (2,3) -- (2,-3);

\draw[fill] (3,3) circle (0.07);
\draw[fill] (-3,3) circle (0.07);
\draw[fill] (3,-3) circle (0.07);
\draw[fill] (-3,-3) circle (0.07);
\draw[fill] (0,0) circle (0.07);
\draw[fill] (3,0) circle (0.07);
\draw[fill] (-3,0) circle (0.07);
\draw[fill] (0,3) circle (0.07);
\draw[fill] (0,-3) circle (0.07);

\draw[fill,blue] (0.75,1.25) circle (0.05);
\draw[fill,blue] (0.75,2.25) circle (0.05);
\draw[fill,blue] (1.75,2.25) circle (0.05);
\draw[blue] (0.75,1.25) -- (1,2);
\draw[blue] (0.75,2.25) -- (1,2);
\draw[blue] (1.75,2.25) -- (1,2);

\draw[fill,blue] (1.25,0.75) circle (0.05);
\draw[fill,blue] (2.25,0.75) circle (0.05);
\draw[fill,blue] (2.25,1.75) circle (0.05);
\draw[blue] (1.25,0.75) -- (2,1);
\draw[blue] (2.25,0.75) -- (2,1);
\draw[blue] (2.25,1.75) -- (2,1);

\draw[fill,blue] (-3+0.75,1.25) circle (0.05);
\draw[fill,blue] (-3+0.75,2.25) circle (0.05);
\draw[fill,blue] (-3+1.75,2.25) circle (0.05);
\draw[blue] (-3+0.75,1.25) -- (-3+1,2);
\draw[blue] (-3+0.75,2.25) -- (-3+1,2);
\draw[blue] (-3+1.75,2.25) -- (-3+1,2);

\draw[fill,blue] (-3+1.25,0.75) circle (0.05);
\draw[fill,blue] (-3+2.25,0.75) circle (0.05);
\draw[fill,blue] (-3+2.25,1.75) circle (0.05);
\draw[blue] (-3+1.25,0.75) -- (-3+2,1);
\draw[blue] (-3+2.25,0.75) -- (-3+2,1);
\draw[blue] (-3+2.25,1.75) -- (-3+2,1);

\draw[fill,blue] (0.75,-3+1.25) circle (0.05);
\draw[fill,blue] (0.75,-3+2.25) circle (0.05);
\draw[fill,blue] (1.75,-3+2.25) circle (0.05);
\draw[blue] (0.75,-3+1.25) -- (1,-3+2);
\draw[blue] (0.75,-3+2.25) -- (1,-3+2);
\draw[blue] (1.75,-3+2.25) -- (1,-3+2);

\draw[fill,blue] (1.25,-3+0.75) circle (0.05);
\draw[fill,blue] (2.25,-3+0.75) circle (0.05);
\draw[fill,blue] (2.25,-3+1.75) circle (0.05);
\draw[blue] (1.25,-3+0.75) -- (2,1-3);
\draw[blue] (2.25,-3+0.75) -- (2,1-3);
\draw[blue] (2.25,-3+1.75) -- (2,1-3);

\draw[fill,blue] (0.75-3,-3+1.25) circle (0.05);
\draw[fill,blue] (0.75-3,-3+2.25) circle (0.05);
\draw[fill,blue] (1.75-3,-3+2.25) circle (0.05);
\draw[blue] (0.75-3,-3+1.25) -- (1-3,-3+2);
\draw[blue] (0.75-3,-3+2.25) -- (1-3,-3+2);
\draw[blue] (1.75-3,-3+2.25) -- (1-3,-3+2);

\draw[fill,blue] (1.25-3,-3+0.75) circle (0.05);
\draw[fill,blue] (2.25-3,-3+0.75) circle (0.05);
\draw[fill,blue] (2.25-3,-3+1.75) circle (0.05);
\draw[blue] (1.25-3,-3+0.75) -- (2-3,1-3);
\draw[blue] (2.25-3,-3+0.75) -- (2-3,1-3);
\draw[blue] (2.25-3,-3+1.75) -- (2-3,1-3);

\draw[red, very thick] (1.25,0.75) -- (0.75,1.25);
\draw[red] (0.75,1.25) -- (-0.75,.75);
\draw[red] (-0.75,0.75) -- (-1.25,-0.75);
\draw[red] (-1.25,-0.75) -- (-0.75,-1.25);
\draw[red] (-.75,-1.25) -- (0.75,-.75);
\draw[red] (0.75,-0.75) -- (1.25,0.75);

\draw[red, very thick] (1+1.25,1.75) -- (1+0.75,2.25);
\draw[red] (0.75,2.25) -- (-0.75,1.75);
\draw[red] (-1-0.75,0.75) -- (-2.25,-0.75);
\draw[red] (-2.25,-1.75) -- (-1.75,-2.25);
\draw[red] (-.75,-2.25) -- (0.75,-1.75);
\draw[red] (1.75,-0.75) -- (2.25,0.75);

\draw[red] (-2+1.25,1.75) -- (-2+0.75,2.25);
\draw[red] (-3+1.25,0.75) -- (-3+0.75,1.25);
\draw[red] (1.75,-0.75) -- (2.25,-1.25);
\draw[red] (0.75,-1.75) -- (1.25,-2.25);

\draw[red] (0.75,2.25) -- (1,3);
\draw[red] (1.75,2.25) -- (2,3);
\draw[red] (-1.25,2.25) -- (-1,3);
\draw[red] (-2.25,2.25) -- (-2,3);
\draw[red] (-2.25,2.25) -- (-3,2);
\draw[red] (-2.25,1.25) -- (-3,1);
\draw[red] (-2.25,-0.75) -- (-3,-1);
\draw[red] (-2.25,-1.75) -- (-3,-2);
\draw[red] (-1.75,-2.25) -- (-2,-3);
\draw[red] (-.75,-2.25) -- (-1,-3);
\draw[red] (1.25,-2.25) -- (1,-3);
\draw[red] (2.25,-2.25) -- (2,-3);
\draw[red] (2.25,-2.25) -- (3,-2);
\draw[red] (2.25,-1.25) -- (3,-1);
\draw[red] (2.25,0.75) -- (3,1);
\draw[red] (2.25,1.75) -- (3,2);

\draw[very thick,dotted] (0.75,1.25) -- (0,0);
\draw[very thick,dotted] (1.25,0.75) -- (0,0);
\draw[very thick,dotted] (-0.75,0.75) -- (0,0);
\draw[very thick,dotted] (-1.25,-0.75) -- (0,0);
\draw[very thick,dotted] (-0.75,-1.25) -- (0,0);
\draw[very thick,dotted] (0.75,-0.75) -- (0,0);

\draw[very thick,dotted, purple] (2.25,0.75) -- (0,0);
\draw[very thick,dotted, purple] (0.75,2.25) -- (0,0);
\draw[very thick,dotted, purple] (1.75,2.25) -- (0,0);
\draw[very thick,dotted, purple] (2.25,1.75) -- (0,0);

\draw[thick, arrows = ->] (3.5,0.5) .. controls (2.5,0.2) and (1.75,0.2) .. (1.5,0.45);
\draw (3.7,0.5) node {$b$};
\draw[thick, arrows = ->] (3.5,-0.5) .. controls (2.5,-0.75) and (1,-0.25) .. (0.75,0.4);
\draw (3.7,-0.5) node {$a$};

\end{tikzpicture}
\caption{The configuration $\Gamma(\Delta_2^{\vee})$ for $g=1$ and $|\bP|=1$, in four fundamental domains. The core spheres $\{a,b\} \subset A_2$ are the vanishing cycles for dotted black respectively purple paths indicated, so the thick red arcs carry corresponding matching spheres. The monodromy along the meridian / longitude exchanges $a$ and $b$.\label{Fig:globalising_configurations}}
\end{center}
\end{figure}

\begin{rmk}
  Recall from Lemma \ref{lem:delta_m_controls_singular_fibres} that the 3-fold associated to the reducible spectral curve \eqref{eqn:reducible_spectral_curve} has isolated singularities of Milnor number $m(m-1)/2$ at the zeroes of $\phi_2$;  the set of $m(m-1)/2$ tripod spheres in the corresponding ideal triangle presumably gives a distinguished basis of vanishing cycles of the singularity (this should follow from \cite{ACampo}, but we will not need it).  \end{rmk}

The total number of Lagrangian spheres in the configuration of tripods and matching spheres is then 
\[
(6g-6+3d)\cdot m + (2g-2+d)\cdot m(m-1).\]
Let $\Gamma(\Delta_m^{\vee})$ denote this set of Lagrangian spheres in $Y_{\Phi}$.

\begin{figure}
\begin{center}
\begin{tikzpicture}[scale=0.7]

\tikzset{cross/.style={cross out, draw=black, minimum size=2*(#1-\pgflinewidth), inner sep=0pt, outer sep=0pt},
cross/.default={4pt}}

\draw[thick] (0,0) node[cross,red] {};
\draw[thick] (-1,0) node[cross,red] {};
\draw[thick] (1,0) node[cross,red] {};
\draw[thick] (.5,.5) node[cross,red] {};
\draw[thick] (-.5,.5) node[cross,red] {};
\draw[thick] (0,1) node[cross,red] {};

\draw[semithick, arrows=->] (0,0) -- ({0+5*cos 100},{0+5*sin 100});
\draw[semithick, arrows=->] (1,0) -- ({1+5*cos 100},{0+5*sin 100});
\draw[semithick, arrows=->] (-1,0) -- ({-1+5*cos 100},{0+5*sin 100});
\draw[semithick, arrows=->] (0.5,0.5) -- ({0.5+5*cos 100},{0.5+5*sin 100});
\draw[semithick, arrows=->] (-0.5,0.5) -- ({-0.5+5*cos 100},{0.5+5*sin 100});
\draw[semithick, arrows=->] (0,1) -- ({0+5*cos 100},{1+5*sin 100});

\draw[semithick, arrows=->] (0,0) -- ({0+5*cos 190},{0+5*sin 190});
\draw[semithick, arrows=->] (1,0) -- ({1+5*cos 190},{0+5*sin 190});
\draw[semithick, arrows=->] (-1,0) -- ({-1+5*cos 190},{0+5*sin 190});
\draw[semithick, arrows=->] (0.5,0.5) -- ({0.5+5*cos 190},{0.5+5*sin 190});
\draw[semithick, arrows=->] (-0.5,0.5) -- ({-0.5+5*cos 190},{0.5+5*sin 190});
\draw[semithick, arrows=->] (0,1) -- ({0+5*cos 190},{1+5*sin 190});

\draw[semithick, arrows=->] (0,0) -- ({0+5*cos 350},{0+5*sin 350});
\draw[semithick, arrows=->] (1,0) -- ({1+5*cos 350},{0+5*sin 350});
\draw[semithick, arrows=->] (-1,0) -- ({-1+5*cos 350},{0+5*sin 350});
\draw[semithick, arrows=->] (0.5,0.5) -- ({0.5+5*cos 350},{0.5+5*sin 350});
\draw[semithick, arrows=->] (-0.5,0.5) -- ({-0.5+5*cos 350},{0.5+5*sin 350});
\draw[semithick, arrows=->] (0,1) -- ({0+5*cos 350},{1+5*sin 350});

\draw[blue] ({0+3*cos 100},{0+3*sin 100}) node {\tiny${12}$};
\draw[blue] ({-1+3*cos 100},{0+3*sin 100}) node {\tiny${12}$};
\draw[blue] ({1+3*cos 100},{0+3*sin 100}) node {\tiny${12}$};
\draw[blue] ({0.5+4*cos 100},{0.5+4*sin 100}) node {\tiny${23}$};
\draw[blue] ({-0.5+4*cos 100},{0.5+4*sin 100}) node {\tiny${23}$};
\draw[blue] ({0+5.25*cos 100},{1+5.25*sin 100}) node {\tiny${34}$};

\draw[blue] ({0+3.5*cos 190},{0+3.5*sin 190}) node {\tiny${32}$};
\draw[blue] ({-1+5.25*cos 190},{0+5.25*sin 190}) node {\tiny${21}$};
\draw[blue] ({1+3*cos 190},{0+3*sin 190}) node {\tiny${43}$};
\draw[blue] ({0.5+3*cos 190},{0.5+3*sin 190}) node {\tiny${43}$};
\draw[blue] ({-0.5+3.5*cos 190},{0.5+3.5*sin 190}) node {\tiny${32}$};
\draw[blue] ({0+3*cos 190},{1+3*sin 190}) node {\tiny${43}$};

\draw[blue] ({0+3.5*cos 350},{0+3.5*sin 350}) node {\tiny${32}$};
\draw[blue] ({-1+3*cos 350},{0+3*sin 350}) node {\tiny${43}$};
\draw[blue] ({1+5.25*cos 350},{0+5.25*sin 350}) node {\tiny${21}$};
\draw[blue] ({0.5+3.5*cos 350},{0.5+3.5*sin 350}) node {\tiny${32}$};
\draw[blue] ({-0.5+3*cos 350},{0.5+3*sin 350}) node {\tiny${43}$};
\draw[blue] ({0+3*cos 350},{1+3*sin 350}) node {\tiny${43}$};


\draw[-stealth,orange,decorate,decoration={snake,amplitude=3pt,pre length=2pt,post length=3pt}] (0,0) --(0.2,-3);
\draw[-stealth,orange,decorate,decoration={snake,amplitude=3pt,pre length=2pt,post length=3pt}] (1,0) --(1.2,-3);
\draw[-stealth,orange,decorate,decoration={snake,amplitude=3pt,pre length=2pt,post length=3pt}] (-1,0) --(-0.8,-3);
\draw[-stealth,orange,decorate,decoration={snake,amplitude=3pt,pre length=2pt,post length=3pt}] (0.5,0.5) --(0.7,-2.5);
\draw[-stealth,orange,decorate,decoration={snake,amplitude=3pt,pre length=2pt,post length=3pt}] (-0.5,0.5) --(-0.3,-2.5);
\draw[-stealth,orange,decorate,decoration={snake,amplitude=3pt,pre length=2pt,post length=3pt}] (0,1) --(0.2,-2);

\end{tikzpicture}
\caption{Branching data for the spectral cover over $\Delta_m$ in one ideal triangle of $\Delta$ (with branch cuts below)\label{Fig:branching_data}}
\end{center}
\end{figure}
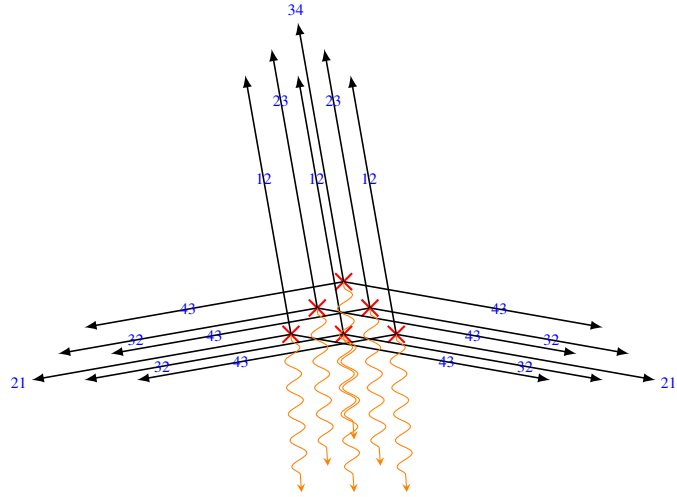

\section{The cyclic potential from holomorphic polygons}

\subsection{Floer theory background}

By construction, $Y_{\Phi}$ is equipped with an integrable complex structure $I$, arising as an algebraic subvariety of an algebraic $\bC^3$-bundle over the Riemann surface $S$. 
We can take its closure in the fibrewise completion, a $\bC\bP^3$-bundle over $S$, to obtain a projective compactification. This is in general singular, but resolving singularities yields a projective compactification $\bar{Y}$ which comes with  a map  $p: \bar{Y} \to S$ to $S$ and has normal crossing boundary. 

We are assuming $g(\bS)>0$, so any rational curve in $\bar{Y}$ maps by a constant map to $S$ so lies in a fibre of $p$, hence meets the boundary.  Since the fibre $p^{-1}(x) \backslash \{p^{-1}(x) \cap Bd(Y)\}$ is affine, the boundary is relatively ample on the singular compactification, hence relatively nef on the resolution.

 For Floer theory, it will be useful to perturb the  complex structure $I$.  We work with the class $\scrJ_{\pi}$ of almost complex structures on $Y$ which tame an $I$-K\"ahler form on $Y$ and which make projection $\pi: Y_{\Phi} \to S$ holomorphic, and which agree with $I$ outside a compact set. In this case, polygons with boundary conditions on the Lagrangians $\Gamma(\Delta_m^{\vee})$ map to holomorphic discs with boundary on the edges of the cellulation $\Delta_m^{\vee}$, to which one can apply the open mapping theorem. Since the fibres of $\pi$ are exact, and contain no rational curves, it is standard that one can achieve transversality in the class $\scrJ_{\pi}$. Although $Y_{\Phi}$ is non-compact and not manifestly of contact type at infinity,  we have:

\begin{lem} \label{lem:compact} The moduli space of holomorphic polygons in $Y_{\Phi}$ with Lagrangian boundary conditions belonging to a compact subset (e.g. to a given finite set of closed Lagrangian submanifolds) is compact.
\end{lem}
\begin{proof} This follows by considering intersections with the singular divisor $Bd(Y) \subset \bar{Y}_{\Phi}$ at infinity.    More precisely, given a sequence $u_j$ of holomorphic discs with boundary conditions in a compact subspace and converging to a stable map $u_{\infty}$, if $u_{\infty}$ does not have image contained in $Y_\Phi$ it must have either a disc component or a rational curve component which meets $Bd(Y)$ but is not wholly contained in the boundary.  Such components meet $Bd(Y)$ strictly positively by positivity of intersection, and relative nefness of $Bd(Y)$ on rational curves (in particular on components contained in the boundary) shows that $u_{\infty} \cdot Bd(Y) > 0$. This contradicts $u_j \cdot Bd(Y) = 0$ for finite $j$. \end{proof}

The Lagrangians we consider are tautologically unobstructed in $Y_{\Phi}$ for almost complex structures making projection $Y_{\Phi} \to S$ holomorphic. Given this, and with compactness from Lemma \ref{lem:compact}, a version of the Fukaya category $\scrF(Y)$ containing Lagrangian matching and tripod spheres can be constructed following the methods of \cite{Seidel:FCPLT}, but working over a Novikov field to take account of convergence issues for holomorphic polygons.  Since all the Lagrangians we consider are spin, and indeed relatively spin for any background class $b \in H^2(Y;\bZ/2)$ supported on the reducible fibres and hence disjoint from $\Gamma(\Delta_m^{\vee})$, we may define $\scrF(Y;b)$ over $\Lambda_{\bC}$.

\subsection{Holomorphic triangles}

Once $m>2$ there are pairs of Lagrangian spheres in the $A_m$-Milnor fibre which are disjoint.  Nonetheless: 

\begin{lem}\label{lem:constant_triangle}
At any vertex $b$ of  $\Delta_m^{\vee}$, the three adjacent Lagrangian spheres  $L_u, L_v, L_w$ meet pairwise transversely at a single point $\hat{b}$ of the corresponding fibre of $Y_{\Phi}$. The constant holomorphic triangle to $\hat{b}$  is regular and contributes to the product $HF(L_v, L_w) \otimes HF(L_u, L_v) \to HF(L_u, L_w)$ (where $L_u, L_v, L_w$ project to arcs ordered clockwise locally at $b$; all three Floer groups are $\bK$).
\end{lem}

\begin{proof}
There is a unique Lefschetz singularity in the $A_m$-fibre lying over a vertex of $\Delta_m^{\vee}$, and locally the three Lagrangians are given by different Lefschetz thimbles near that point. (The isotopy in the construction of the matching spheres which makes them only essentially fibred can be taken to be supported away from the common critical end-point.)  It follows that they meet pairwise transversely, and indeed can be locally described by three linear Lagrangian subspaces in $\bC^n$. Such a triple of Lagrangian subspaces meeting at a point can be modelled on a product of copies of three real lines in $\bC$ meeting at the origin.  Regularity of the constant map for the correct cyclic order is standard, see \cite[Lemma 4.9]{Smith:quiver}.
\end{proof}

For $m>2$ there are internal `white' triangles in the quiver $Q(\Delta_m)$, i.e. primitive 3-cycles of the form $q_w$, which contribute regions to the Lagrangian cellulation $\Delta_m^{\vee}$ which have three tripod Lagrangian boundary components (cf. the internal `hexagons' in Figure \ref{Fig:a3_lag_cellulation},  note these have only three geometrically distinct Lagrangian boundaries).

\begin{lem} \label{lem:more} Suppose $m>2$. 
Consider a primitive 3-cycle $q_w$ with tripod Lagrangian boundaries $L_u, L_v, L_w$ in cyclic order. The Floer product $HF^1(L_v, L_w) \otimes HF^1(L_u,L_v) \to HF^2(L_u,L_w)$ is non-zero. 
\end{lem}

\begin{proof} Since only four sheets of the spectral cover are involved in the geometry of Figure \ref{Fig:a3_lag_cellulation}, it suffices to consider the case $m=3$. By Lemma \ref{lem:tripod_vs_matching_sphere}, the three tripod spheres can be replaced by matching spheres for the local structure of $Y_{\Phi}$ as an $A_3$-fibred Lefschetz fibration.  The matching paths can moreover be taken to meet at a unique point,  the centre of the white triangle, see Figure  \ref{Fig:force_constant_triangle}. Then  the corresponding 3-spheres meet only in the fibre over that point (more precisely, this is true after a symplectomorphism of an exact subdomain containing the given triple of spheres, after which they can be taken to be exactly fibred over arcs in the base).   The Hamiltonian isotopies from tripods to matching spheres induce quasi-isomorphisms of the corresponding objects in the Fukaya category, and one can arrange that there is no wall-crossing since the isotopies are through weakly exact Lagrangians. We need to determine the vanishing 2-spheres in the smooth $A_3$-fibre $F$ over the central point in Figure \ref{Fig:force_constant_triangle}.  

From the original configuration of the tripods, these 2-spheres meet pairwise with rank one Floer cohomology. By \cite{Khovanov-Seidel} this is only possible for matching spheres in $A_k$ if the spheres are fibred over paths which pairwise share a single end-point (so meet geometrically once).   This reduces us to the geometry which entered into the discussion of the constant triangle in Lemma \ref{lem:constant_triangle}.
\end{proof}

\begin{figure}
\begin{center}
\begin{tikzpicture}[scale=1.25]

\newcommand{\midarrow}{}

\draw[semithick,dashed] (0,0) -- (4,0);
\draw[semithick,dashed] (0,0) -- ({4*cos(60)}, {4*sin(60)});
\draw[semithick,dashed] ({4*cos 60}, {4*sin 60}) -- (4,0);

\draw[semithick,dashed] (1,0) -- node {\midarrow} ({1+cos 60}, {sin 60});
\draw[semithick,dashed] ({1+cos 60}, {sin 60}) -- node {\midarrow} ({cos 60}, {sin 60});
\draw[semithick,dashed] ({cos 60}, {sin 60}) -- node {\midarrow} (1,0);

\draw[semithick,dashed] (2,0) -- node {\midarrow} ({2+cos 60}, {sin 60});
\draw[semithick,dashed] ({2+cos 60}, {sin 60}) -- node {\midarrow} ({1+cos 60}, {sin 60});
\draw[semithick,dashed] ({1+cos 60}, {sin 60}) -- node {\midarrow} (2,0);

\draw[semithick,dashed] (3,0) -- node {\midarrow} ({3+cos 60}, {sin 60});
\draw[semithick,dashed] ({3+cos 60}, {sin 60}) -- node {\midarrow} ({2+cos 60}, {sin 60});
\draw[semithick,dashed] ({2+cos 60}, {sin 60}) -- node {\midarrow} (3,0);


\draw[semithick,dashed] ({1+cos 60}, {sin 60}) -- node {\midarrow} ({1+2*cos 60}, {2*sin 60});
\draw[semithick,dashed] ({1+2*cos 60}, {2*sin 60}) -- node {\midarrow} ({2*cos 60}, {2*sin 60});
\draw[semithick,dashed] ({2*cos 60}, {2*sin 60}) -- node {\midarrow} ({1+cos 60},{sin 60});

\draw[semithick,dashed] ({2+cos 60}, {sin 60}) -- node {\midarrow} ({2+2*cos 60}, {2*sin 60});
\draw[semithick,dashed] ({2+2*cos 60}, {2*sin 60}) -- node {\midarrow} ({1+2*cos 60}, {2*sin 60});
\draw[semithick,dashed] ({1+2*cos 60}, {2*sin 60}) -- node {\midarrow} ({2+cos 60},{sin 60});

\draw[semithick,dashed] (2, {2*sin 60}) -- node {\midarrow} ({2+cos 60}, {3*sin 60});
\draw[semithick,dashed] ({2+cos 60}, {3*sin 60}) -- node {\midarrow} ({1+cos 60}, {3*sin 60});
\draw[semithick,dashed] ({1+cos 60}, {3*sin 60}) -- node {\midarrow} ({2},{2*sin 60});

\draw[fill,blue] (1,{0.5*sin 60}) circle (0.05);
\draw[fill,blue] (2,{0.5*sin 60}) circle (0.05);
\draw[fill,blue] (3,{0.5*sin 60}) circle (0.05);
\draw[fill,blue] (1.5,{1.5*sin 60}) circle (0.05);
\draw[fill,blue] (2.5,{1.5*sin 60}) circle (0.05);
\draw[fill,blue] (2,{2.5*sin 60}) circle (0.05);

\draw[fill,red] (2,{1.25*sin 60}) circle (0.03);

\draw[semithick, blue] (1,{0.5*sin 60}) -- (1.5,{sin 60});
\draw[semithick, blue] ({2, 0.5*sin 60}) -- (1.5,{sin 60});
\draw[semithick, blue] ({1.5,1.5*sin 60}) -- (1.5,{sin 60});

\draw[semithick, blue] (2,{0.5*sin 60}) -- (2.5,{sin 60});
\draw[semithick, blue] ({3, 0.5*sin 60}) -- (2.5,{sin 60});
\draw[semithick, blue] ({2.5,1.5*sin 60}) -- (2.5,{sin 60});

\draw[semithick, blue] (1.5,{1.5*sin 60}) -- (2,{2*sin 60});
\draw[semithick, blue] ({2.5, 1.5*sin 60}) -- (2,{2*sin 60});
\draw[semithick, blue] ({2,2.5*sin 60}) -- (2,{2*sin 60});

\draw[semithick, red, rounded corners] (1, {0.5*sin 60}) -- (1,1.25) -- (1.5,1.65) -- (2,1.25) -- (2,{0.5*sin 60});

\draw[semithick, red, rounded corners] (1.5,{1.5*sin 60}) -- (2, {1.25*sin 60}) -- (3,{1.25*sin 60}) -- (2.5,{2.25*sin 60}) -- (2,{2.5*sin 60});

\draw[semithick, red, rounded corners] (2.5, {1.5*sin 60}) -- (2,{1.25*sin 60}) -- (1.5,{0.25*sin 60}) -- (2.5, {0.25*sin 60}) -- (3,{0.5*sin 60});

\end{tikzpicture}
\caption{Deforming a triple of tripod spheres (blue) to matching spheres (red) which meet at a point, defining a constant holomorphic triangle\label{Fig:force_constant_triangle}}
\end{center}
\end{figure}
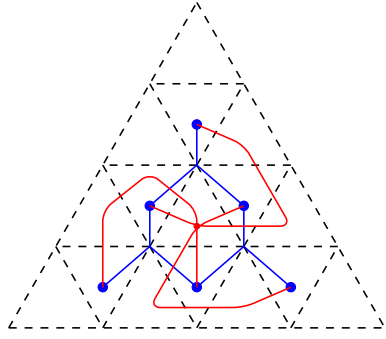

\subsection{Geometry near the reducible fibre}

We briefly recall the geometry near the reducible fibre.  After deformation, there is a local model for $Y$
\begin{equation} \label{eqn:reducible_local_model}
\{\delta(a^2 + c^2) + \prod_{j=1}^m (b-j) = 0\} \subset \bC^3\times \bC_{\delta}
\end{equation}
in which the projection $Y\to S$ is modelled on projection to the $\delta$-plane. The general fibres are type $A_m$-surfaces $b^m + O(b^{m-1}) + (\mathrm{const.})(a^2+c^2) = 0$, whilst the fibre over $\delta =0$ is given by the $m$-tuple of planes $\{b=j\} \times \bC^2_{a,c}$, for $j \in \{1,\ldots, m\}$. There is a Lagrangian boundary condition
\[
\delta = e^{i\theta}, \ a,c \in e^{-i\theta/2}\cdot (-1)^m i \cdot \bR, \ b \in [j,j+1]
\]
for each $1\leq j\leq m-1$, defining a totally real $S^1\times S^2$ lying over the unit circle in the $\delta$-plane. One can deform the standard symplectic structure in a neighbourhood of this totally real submanifold to make it Lagrangian, and fixed by an antiholomorphic involution, cf. \cite[Section 4.6]{Smith:quiver}.  The only holomorphic discs with boundary on this Lagrangian cylinder are given by (multiple covers of) the constant sections over the unit disc in the $\delta$-plane:
\begin{equation} \label{eqn:two_discs}
u: (D, \partial D) \to (Y, (S^1\times S^2)_j), \quad u(z) = (0,\ast,0, z), \ \ast \in \{j,j+1\}.
\end{equation}

There is another viewpoint which can be helpful. There is a unitary local change of co-ordinates which transforms the  local model \eqref{eqn:reducible_local_model} to the form
\[
\{\delta\cdot uv + \prod_j (b-j) = 0\} \subset \bC^4
\] 
and one can consider projection to the $b$-plane. The generic fibre is now $\{\delta u v = \mathrm{const}\} \subset \bC^3$, which is a copy of $(\bC^*)^2$; the fibres over $b=j$ are isomorphic to the union of the three co-ordinate planes $\delta u v = 0 \subset \bC^3$. The local model of the map $xyz: \bC^3\to \bC$ has been studied extensively  in \cite{AAK}. Again consider the matching path $[j,j+1] \subset \bR$ between two critical values of the projection.  One can parallel transport the Lagrangian $T^2 \subset T^*T^2 = (\bC^*)^2$ along this path, to obtain a Lagrangian $S^1\times S^2$ which is another model for that considered above. The two holomorphic discs with boundary on the Lagrangian now lie entirely over the end-points: the Lagrangian meets the fibre over $j$ in the unit circle in the $\delta$-plane, and bounds the obvious disc lying entirely in the singular locus of the $j$-fibre.  Note that from the second viewpoint, there are three Lagrangian $(S^1\times S^2)$'s associated to $[j,j+1]$, given the symmetry in the co-ordinates $\delta,u,v$; only one of these is fibred with respect to the $\delta$-plane projection.

The second viewpoint makes it especially clear that the holomorphic discs with boundary on $S^1\times S^2$ have vanishing Maslov class, by comparing to the toric model $xyz: \bC^3 \to \bC$. 

\begin{lem} Given a choice of spin structure on $L_j \cong S^1\times S^2$ and hence orientation of the moduli space of rigid discs with boundary on $L_j$, the two holomorphic discs from \eqref{eqn:two_discs} have opposite sign.
\end{lem}

\begin{proof} The geometry is local near the given $S^1\times S^2$, and the argument from \cite[Section 4.6]{Smith:quiver} applies.
\end{proof}

\begin{defn} \label{defn:eigen} An \emph{eigen-ordering} of a generic tuple $\Phi$ is a choice of ordering of the roots of $\Phi(b) = 0$ near each point $p\in D$.  One can equivalently think of an eigen-ordering as giving an ordering of the irreducible components of the fibre $(Y_{\Phi})_p$ over each point $p\in D \subset S$.\end{defn}

The space of eigen-ordered generic tuples is an unramified $\Sym_m^{\times d}$ cover of an open subset of the Hitchin base.  In analogy with \cite{BridgelandSmith}, one expects eigen-ordered generic tuples to define stability conditions on the category $\scrC$. 

Note that each one of the discs meets exactly one component of the reducible fibre of $Y$ over $0 \in \bC_{\delta}$. Moreover, consideration of the branching behaviour in Figure \ref{Fig:branching_data} shows that as one varies the value $j$ when considering sections over $L_p^{(j)}$, different pairs of irreducible components of the fibre over $p \equiv 0 \in \bC_{\delta}$ meet the corresponding holomorphic discs; compare to the final co-ordinate in \eqref{eqn:two_discs}.  Given an ordering of the components of the fibres over $\bP$, there is an associated choice of cycle $Z_b$ comprising $\lceil (m+1)/2 \rceil$ of the irreducible components for which the total signed intersection number of the discs of \eqref{eqn:two_discs} with $Z_b$ is necessarily non-zero (because only one of each pair of discs hits one of the components included in $Z_b$).    Up to monodromy in the space of eigen-ordered tuples, which induces a permutation representation of the components of the reducible fibres, we can assume that this choice of cycle is just the even-indexed components as specified in the Introduction (which is a suitable cycle choice if the local geometry agrees with the labelling of sheets from Figure \ref{Fig:branching_data}).

\subsection{Holomorphic discs for other primitive cycles}

Fix an ideal triangulation $\Delta$ and the collection $\{L_v \, | \, v \in \Gamma(\Delta_m^{\vee})\}$ of Lagrangian spheres associated to the dual Lagrangian cellulation. We wish to understand the holomorphic discs which contribute to the $A_{\infty}$-structure on $\scrA_{\Gamma}$.  There are constant holomorphic triangles indexed by the primitive cycles $c_b$ of the quiver, which we have already encountered. The other two classes of primitive cycle  also give rise to holomorphic discs. 

\begin{prop} \label{prop:rigid_discs}
Fix a vertex $p \in \bP$ of $\Delta$. For each $1 \leq j \leq m$, the moduli space of rigid holomorphic discs with boundary $L_p^{(j)}$ is non-empty. Moreover, there is a choice of cycle representative for the background class $b \in H^2(Y_{\Phi};\bZ/2)$ of \eqref{eqn:background} for which the algebraic count of such discs is non-zero.
\end{prop}

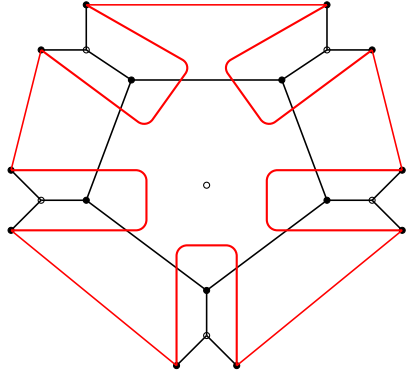
\begin{figure}[ht]
\begin{center}
\begin{tikzpicture}[scale =0.8]

\draw[fill] (-1.25,0) circle (0.05);
\draw[fill] (1.25,0) circle (0.05);
\draw[fill] (-2,-2) circle (0.05);
\draw[fill] (2,-2) circle (0.05);
\draw[fill] (0,-3.5) circle (0.05);
\draw (2,0.5) circle (0.05);
\draw (-2,0.5) circle (0.05);
\draw[fill] (2.75,0.5) circle (0.05);
\draw[fill] (-2.75,0.5) circle (0.05);
\draw[fill] (2,1.25) circle (0.05);
\draw[fill] (-2,1.25) circle (0.05);

\draw (2.75,-2) circle (0.05);
\draw[fill] (3.25,-2.5) circle (0.05);
\draw[fill] (3.25,-1.5) circle (0.05);
\draw (-2.75,-2) circle (0.05);
\draw[fill] (-3.25,-2.5) circle (0.05);
\draw[fill] (-3.25,-1.5) circle (0.05);

\draw (0,-4.25) circle (0.05);
\draw[fill] (-.5,-4.75) circle (0.05);
\draw[fill] (0.5,-4.75) circle (0.05);

\draw (0,-1.75) circle (0.05);

\draw[semithick] (-1.25,0) -- (1.25,0) -- (2,-2) -- (0,-3.5) -- (-2,-2) -- cycle;
\draw[semithick] (1.25,0) -- (2, 0.5);
\draw[semithick] (2,0.5) -- (2.75, 0.5);
\draw[semithick] (2,0.5) -- (2,1.25);
\draw[semithick] (-1.25,0) -- (-2, 0.5);
\draw[semithick] (-2,0.5) -- (-2.75, 0.5);
\draw[semithick] (-2,0.5) -- (-2,1.25);

\draw[semithick] (2,-2) -- (2.75,-2);
\draw[semithick] (2.75,-2) -- (3.25,-1.5);
\draw[semithick] (2.75,-2) -- (3.25,-2.5);
\draw[semithick] (-2,-2) -- (-2.75,-2);
\draw[semithick] (-2.75,-2) -- (-3.25,-1.5);
\draw[semithick] (-2.75,-2) -- (-3.25,-2.5);

\draw[semithick] (0,-4.25) -- (0,-3.5);
\draw[semithick] (0,-4.25) -- (-.5,-4.75);
\draw[semithick] (0,-4.25) -- (.5,-4.75);

\draw[semithick, red] (2.75,0.5) -- (3.25,-1.5);
\draw[semithick, red] (-2.75,0.5) -- (-3.25,-1.5);
\draw[semithick, red] (.5,-4.75) -- (3.25,-2.5);
\draw[semithick, red] (-.5,-4.75) -- (-3.25,-2.5);
\draw[semithick, red] (2,1.25) -- (-2,1.25);

\draw[thick,red,rounded corners] (2,1.25) -- (0.25,.25) -- (1,-0.8) -- (2.75,0.5);
\draw[thick,red,rounded corners] (-2,1.25) -- (-0.25,.25) -- (-1,-.8) -- (-2.75,0.5);

\draw[thick,red,rounded corners] (3.25,-1.5) -- (1,-1.5) -- (1,-2.5) -- (3.25,-2.5);
\draw[thick,red,rounded corners] (-3.25,-1.5) -- (-1,-1.5) -- (-1,-2.5) -- (-3.25,-2.5);
\draw[thick,red,rounded corners] (0.5,-4.75) -- (0.5,-2.75) -- (-0.5,-2.75) -- (-.5,-4.75);

\end{tikzpicture}
\end{center}
\caption{Reducing the disc count over $L_p^{(2)}$ to one analogous to that over $L_p^{(1)}$, compare to Figure \ref{Fig:tripod_cycle}\label{Fig:second_tier_after_isotopy}}

\end{figure}

\begin{proof}
The argument for counting discs over $L_p^{(1)}$ is almost the same as in \cite{Smith:quiver}, relying on a degeneration technique to reduce to the count of discs on a Lagrangian $S^1\times S^2$, as found in \eqref{eqn:two_discs},  and the behaviour of holomorphic discs under Lagrange surgery from \cite[Chapter 10]{FO3}.  For the higher $L_p^{(j)}$, there is a trick to reduce to the computation to the previously studied case, indicated schematically in Figure \ref{Fig:second_tier_after_isotopy} in the case $j=2$. Namely, if one replaces the tripod spheres by their Hamiltonian deformations as shown in red in the figure, then one reduces to the case of a region of the base $S$ containing a single point of $\bP$ and with boundary a polygon of matching paths, each labelled by the same Lagrangian vanishing cycle in the fibre. This is exactly the situation of the disc count over $L_p^{(1)}$, except the particular components of the reducible fibre which the holomorphic sections intersect will depend on $j$, compare to the discussion at the end of the previous section.  Note that when $j>2$ (which arises only when $m>2$), the boundary configuration of $L_p^{(j)}$  involves adjacent tripod spheres, and not only alternating tripod and matching spheres.  However, this doesn't affect the argument. \end{proof}

It would be reasonable to expect that there is a holomorphic 3-form on $Y_{\Phi}$ with respect to which all the Lagrangian spheres in the configuration $\Gamma(\Delta_m^{\vee})$ admit gradings making them special of phase zero.  At least the topological analogue of this holds:

\begin{lem} 
One can grade the Lagrangians in the configuration $\Gamma(\Delta_m^{\vee})$ consistently so that all polygons in the cellulation have Maslov index zero, and the Floer algebra is concentrated in degrees $0\leq \ast \leq 3$.
\end{lem}

\begin{proof} Fix an element $p\in \bP$ and grade all but one of the Lagrangians encircling $p$ -- the boundaries of a polygon projecting to $L_p^{(1)}$ -- so that their intersections have degree $1$ as Floer inputs. The existence of the rigid disc of Proposition \ref{prop:rigid_discs} implies that the Floer output has degree 2, so the gradings are in fact cyclically symmetric.  The existence of the rigid polygons over $L_p^{(j)}$ with $j \geq 1$,   together with the fact that every matching sphere belongs to the boundary of a unique $L_p^{(j)}$, imply that the gradings propagate consistently to yield a grading satisfying the  required conditions.  (It follows by additivity of Maslov index that the gradings are consistent with the existence of the rigid quadrilaterals constructed in Proposition \ref{prop:rigid_quadrilateral} below.)
\end{proof}

We now fix the background class $b\in H^2(Y_{\Phi};\bZ/2)$ which is Poincar\'e dual to a four-cycle $Z_b$ defined by `half' the irreducible components at all the singular fibres. Precisely,
\begin{equation} \label{eqn:background}
b = PD[Z_b], \quad Z_b = \sum_{p \in \bP} \bC^2_{(ev)}
\end{equation}
where the fibre $\pi^{-1}(p) \subset Y_{\Phi}$ of $p: Y_{\Phi} \to S$ is a disjoint union of $(m+1)$ ordered copies of $\bC^2$, the union of the even-indexed components of which we have labelled $\bC^2_{(ev)}$. If there are no 3-valent vertices in the ideal triangulation, one can compute the endomorphism algebra $\scrA_{\Gamma}$ equivalently by working either in $\scrF(Y)$ or in $\scrF(Y;b)$; but in the presence of 3-valent vertices and $L_p^{(1)}$ triangles, twisting by $b$ potentially affects the cohomological algebra.

\begin{prop}\label{prop:algebra_ok}
The algebra $\scrA_{\Gamma} := \oplus_{v, v' \in \Gamma(\Delta_m^{\vee})} \, HF^*(L_v, L_{v'})$ is isomorphic to the total endomorphism algebra of the category $\scrC(Q(\Delta_m), W_{\bf c}(\Delta_m))$ for a vector ${\bf c}$ of non-zero coefficients.
\end{prop}

\begin{proof}
There are three types of holomorphic triangles which contribute to the Floer product in the algebra: 
\begin{enumerate}
\item  constant triangles with image a vertex of $\Delta_m^{\vee}$; 
\item  the triangles of Lemma \ref{lem:more};
\item triangles which map to a cycle $L_p^{(1)}$ for a vertex $p$ of $\Delta$ of valence 3 (if any such exist).
\end{enumerate}
Each of these three triangle types has a non-zero count,  and all the corresponding terms appear in the potential $W_{\bf c}(\Delta_m)$. Working over $\Lambda_{\bC}$, we  take the coefficients of the first set of triangles to be $+1$, the second set of triangles to be of lowest order valuation (since the triangles become constant only after a Hamiltonian isotopy of the spheres in the configuration). The coefficients in ${\bf c}$ for a triangular $L_p^{(1)}$-region $R$  will be `large', in the sense that it will be counted by $q^{\langle [\omega], R\rangle}$.
\end{proof}

\begin{prop} \label{prop:rigid_quadrilateral}
For a primitive cycle associated to a white quadrilateral $b_w$  the corresponding count of holomorphic discs is non-trivial.
\end{prop}

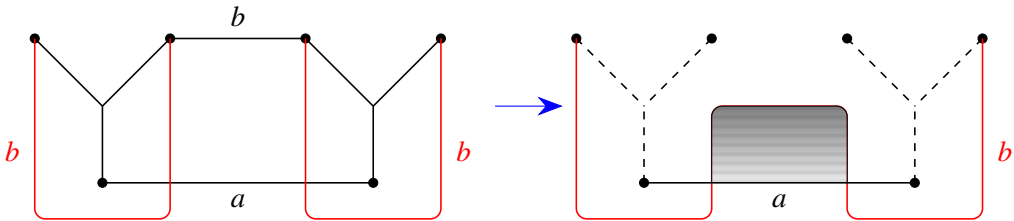
\begin{figure}[ht]
\begin{center}
\begin{tikzpicture}[scale = 0.6]

\draw[fill] (0-3,-1.7) circle (0.1);
 \draw[fill] (1.5-3,1.5) circle (0.1);
 \draw[fill] (-1.5-3,1.5) circle (0.1);
 \draw[semithick] (0-3,-1.7) -- (0-3,0);
 \draw[semithick] (1.5-3,1.5) -- (0-3,0);
 \draw[semithick] (-1.5-3,1.5) --(0-3,0);

\draw[fill] (0-9,-1.7) circle (0.1);
 \draw[fill] (1.5-9,1.5) circle (0.1);
 \draw[fill] (-1.5-9,1.5) circle (0.1);
 \draw[semithick] (0-9,-1.7) -- (0-9,0);
 \draw[semithick] (1.5-9,1.5) -- (0-9,0);
 \draw[semithick] (-1.5-9,1.5) --(0-9,0);
 
 \draw[semithick] (1.5-9,1.5) -- (-1.5-3,1.5);
 \draw[semithick] (-9,-1.7) -- (-3,-1.7);
 \draw (-6,2) node {$b$};
 \draw (-6,-2.1) node {$a$};
 \draw[semithick, red, rounded corners] (-1.5-9,1.5) -- (-1.5-9,-2.5) -- (1.5-9,-2.5) -- (1.5-9,1.5);
  \draw[semithick, red, rounded corners] (-1.5-3,1.5) -- (-1.5-3,-2.5) -- (1.5-3,-2.5) -- (1.5-3,1.5);
\draw[red]  (-11,-1) node {$b$};
 \draw[red]  (-1,-1) node {$b$};

\draw[blue,-{Stealth[scale=2.0]}] (-0.3,0) -- (1.2,0);
 
\draw[fill] (0+3,-1.7) circle (0.1);
 \draw[fill] (1.5+3,1.5) circle (0.1);
 \draw[fill] (-1.5+3,1.5) circle (0.1);
 \draw[semithick, dashed] (0+3,-1.7) -- (0+3,0);
 \draw[semithick, dashed] (1.5+3,1.5) -- (0+3,0);
 \draw[semithick, dashed] (-1.5+3,1.5) --(0+3,0);
 
\draw[fill] (0+9,-1.7) circle (0.1);
 \draw[fill] (1.5+9,1.5) circle (0.1);
 \draw[fill] (-1.5+9,1.5) circle (0.1);
 \draw[semithick,dashed] (0+9,-1.7) -- (0+9,0);
 \draw[semithick,dashed] (1.5+9,1.5) -- (0+9,0);
 \draw[semithick,dashed] (-1.5+9,1.5) --(0+9,0);
 
 \draw (6,-2.1) node {$a$};
 \draw[semithick,red,rounded corners] (1.5,1.5) -- (1.5,-2.5) -- (4.5,-2.5) -- (4.5,0)  -- (7.5,0) -- (7.5,-2.5) -- (10.5,-2.5) -- (10.5,1.5); 
 \draw[red]  (11,-1) node {$b$};
 \begin{scope}
 \clip (4.5,-1.7) rectangle  ++(3,3);
\draw[fill,shade, rounded corners] (4.5,-2) -- (4.5,0) -- (7.5,0) -- (7.5,-2) -- (4.5,-2);
 \end{scope}
   \draw[semithick] (9,-1.7) -- (3,-1.7);

\end{tikzpicture}
\end{center}
\caption{The count of discs over a primitive quadrilateral region $b_w$ is non-zero\label{Fig:quadrilateral_discs}}
\end{figure}

\begin{proof}
See Figure \ref{Fig:quadrilateral_discs}.  The count of holomorphic discs is obtained by an invertible continuation isomorphism from the count in which the tripod boundary conditions are replaced by their red Hamiltonian images.  The relation between holomorphic discs and polygons before and after Lagrange surgery \cite[Chapter 10]{FO3} relates this to the count of holomorphic strips on the right hand side of the Figure. In this picture, the two Lagrangian 3-sphere boundary conditions are Hamiltonian disjoinable; the Floer complex has total rank two, with exactly one intersection point lying over each intersection of the black and red curves in the image.  The Floer differential must be non-trivial (and hence an isomorphism). Putting one marked point in the interior of the strip $\bR\times [0,1]$ to stabilise the domain, since we are counting sections of the fibration $Y_{\Phi} \to S$ over the quadrilateral, it follows that the count of holomorphic quadrilaterals over the original domain $b_w$ is algebraically $\pm 1$.
\end{proof}

\begin{cor} The $A_{\infty}$-structure on $\scrA_{\Gamma}$ is encoded by a generic potential, i.e. one of the form $W = W_{\bf c}(\Delta_b) + W'$ for some ${\bf c} \in (\bK^*)^N$ and $W'$ concentrated on non-primitive cycles.
\end{cor}

\begin{proof} The previous lemmata show that the coefficients of all the primitive cycles are non-zero; in the case of the $L_p^{(j)}$ this relies on twisting by the background class $b$ to ensure that, for each $j$,  the two contributing holomorphic discs (which have the same area) cannot cancel.
\end{proof}

\begin{rmk} \label{rmk:class_of_omega}  Lemma \ref{lem:H2Y} implies that the cohomology class of a K\"ahler form $[\omega] \in H^2(Y_{\Phi};\bR)$ is  determined by the total area of the base $S$ and its evaluation on a collection of $m$ closed surfaces at each reducible fibre whose intersection matrix with the irreducible components has rank $m$.  Such a collection of surfaces can be obtained as follows: at a point of $D$, fix some $1 \leq j \leq m$, and consider the two holomorphic discs lying over $L^{(j)}_p$.  Interpolating their boundaries fibrewise inside a component of the core $A_m$-chain of spheres in the fibres gives a 2-sphere meeting exactly two of the irreducible components. The symplectic area of such a sphere is just twice the coefficient in the potential of the primitive cycle $L_p^{(j)}$.  It follows that the map $H^2(Y_{\Phi};\bR) \supset U \to \{\mathrm{primitive \ potentials}\}$ is locally injective. On the other hand, if we divide out by gauge equivalence, then one can normalise the coefficient of $L_p^{(1)}$ to be $1$, so the `mirror map' $U \to \{\mathrm{potentials}\}/\{\mathrm{gauge}\}$ is not injective even after factoring out global rescaling of $[\omega]$. \end{rmk}

\begin{rmk}  \label{rmk:relative}
One could consider the subcategory $\scrA_{\Gamma}$  inside the `relative Fukaya category' $\scrF(Y_{\Phi},\scrD;b)$,  see \cite{Sheridan:CDM},  for $\scrD\subset Y_{\Phi}$ the divisor given by the union of fibres over $D \subset S$.  The coefficients of holomorphic polygons  in the relative Fukaya category record intersection numbers with  $\scrD$; one can then take all the coefficients for $L_p^{(j)}$ equal. Similarly, if one works with a monodromy-invariant K\"ahler form as in Remark \ref{rmk:kahler_monodromy}, then the coefficients of all the $L_p^{(j)}$ in the potential will be some fixed power $q^a$ of the Novikov variable, which brings one closer to the `canonical' potential. 
\end{rmk}

\bibliographystyle{amsalpha}
{\footnotesize{\bibliography{mybib}}}

\end{document}